\documentclass[11pt]{amsart}

\usepackage{slashbox}

\usepackage{amsfonts,epsfig}
\usepackage{latexsym}
\usepackage{amssymb}
\usepackage{amsmath}
\usepackage{amsthm}
\usepackage{graphics}
\usepackage[all]{xy}
\usepackage[T2A]{fontenc}
\usepackage{multirow}
\usepackage{array, booktabs, ctable}
\usepackage{enumerate}
\usepackage{booktabs}
\usepackage{longtable}
\usepackage{color}
\usepackage{multirow}
\usepackage[pagebackref]{hyperref}
\usepackage{xcolor,colortbl}

\renewcommand*{\arraystretch}{1.3}

\newtheorem{theorem}{Theorem}
\newtheorem{corollary}[theorem]{Corollary}

\newtheorem{proposition}[theorem]{Proposition}

\newtheorem*{ack}{Acknowledgements}

\theoremstyle{definition}\newtheorem*{remark}{Remark}
\newtheorem{example}{Example}[section]

\newcommand{\cC}{{\mathcal{C}}}
\DeclareMathOperator{\cm}{CM} 
\DeclareMathOperator{\tors}{tors}

\newcommand{\Q}{\mathbb Q}
\newcommand{\Qbar}{{\overline{\mathbb Q}}} 
\newcommand{\Z}{\mathbb Z}

\newcommand{\Gal}{\operatorname{Gal}}

\newcommand{\GL}{\operatorname{GL}}

\begin{document}

\bibliographystyle{plain}
\title[Torsion of rational elliptic curves over quartic fields]{On the torsion of rational elliptic curves\\ over quartic fields}

\author{Enrique Gonz\'alez--Jim\'enez}
\address{Universidad Aut{\'o}noma de Madrid, Departamento de Matem{\'a}ticas, Madrid, Spain}
\email{enrique.gonzalez.jimenez@uam.es}
\author{\'Alvaro Lozano-Robledo}
\address{University of Connecticut, Department of Mathematics, Storrs, CT 06269, USA}
\email{alvaro.lozano-robledo@uconn.edu} 

\thanks{The first author was partially  supported by the grant MTM2015--68524--P}

\date{\today}
\begin{abstract}
Let $E$ be an elliptic curve defined over $\Q$ and let $G = E(\Q)_{\tors}$ be the associated torsion subgroup. We study, for a given $G$, which possible groups $G \subseteq H$ could appear such that $H=E(K)_{\tors}$, for $[K:\Q]=4$ and $H$ is one of the possible torsion structures that occur infinitely often as torsion structures of elliptic curves defined over quartic number fields.
\end{abstract}

\maketitle

Let $K$ be a number field, and let $E$ be an elliptic curve over $K$. The Mordell-Weil theorem states that the set $E(K)$ of $K$-rational points on $E$ is a finitely generated abelian group. It is well known that $E(K)_{\tors}$, the torsion subgroup of $E(K)$, is isomorphic to $\Z/n\Z\times\Z/m\Z$ for some positive integers $n,m$ with $n|m$. In the rest of the paper we shall write $\cC_n=\Z/n\Z$ for brevity, and we call $\cC_n\times\cC_m$ the torsion structure of $E$ over $K$.

The characterization of the possible torsion structures of elliptic curves has been of considerable interest over the last few decades. Since Mazur's proof \cite{Mazur1978} of Ogg's conjecture,\footnote{See \cite{Schappacher} that stablished all the torsion structures over the rationals for a discussion of the authorship of this conjecture} and Merel's proof \cite{Merel} of the uniform boundedness conjecture, there have been several interesting developments in the case of a number field $K$ of fixed degree $d$ over $\Q$. The case of quadratic fields ($d=2$) was completed by Kamienny \cite{K92}, and Kenku and Momose \cite{KM88} after a long series of papers. However, there is no complete characterization of the torsion structures that may occur for any fixed degree $d>2$ at this time.\footnote{See \cite{Sutherland} for a very nice survey on the subject.} Nevertheless, there has been significant progress to characterize the cubic case \cite{JKP04,JKL11,N12,J15,BN15,W15} and the quartic case \cite{JKP06,JKL13,JKL15,N12b}.   Let us define some useful notations to describe more precisely what is known for $d\geq 2$:
\begin{itemize}
\item Let $S(d)$ be the set of primes that can appear as the order of a torsion point of an elliptic curve defined over a number field of degree $\le d$.
\item Let $\Phi(d)$ be the set of possible isomorphism torsion structures $E(K)_{\tors}$, where $K$ runs through all number fields $K$ of degree $d$ and $E$ runs through all elliptic curves over $K$. 
\item Let $\Phi^{\infty}(d)$ be the subset of isomorphic torsion structures in $\Phi(d)$  that occur infinitely often. More precisely, a torsion structure $G$ belongs to  $\Phi^{\infty}(d)$ if there are infinitely many elliptic curves $E$, non-isomorphic over $\overline{\Q}$, such that $E(K)_{\tors}\simeq G$.
\end{itemize}
Mazur established that $S(1)=\{2,3,5,7\}$ and
$$
\Phi(1) = \left\{ \cC_n \; | \; n=1,\dots,10,12 \right\} \cup \left\{ \cC_2 \times \cC_{2m} \; | \; m=1,\dots,4 \right\}.
$$
Kamienny, Kenku and Momose established that  $S(2)=\{2,3,5,7,11,13\}$ and
$$
\Phi(2) = \left\{ \cC_n \; | \; n=1,\dots,16,18 \right\} \cup \left\{ \cC_2 \times \cC_{2m} \; | \; m=1,\dots,6 \right\}  \cup 
\left\{ \cC_3 \times \cC_{3r} \; | \; r=1,2 \right\} \cup \left\{ \cC_4 \times \cC_4 \right\}.  
$$
The elliptic curves with torsion structure $\cC_n\times\cC_m$ are parametrized by the modular curve $X_1(n,m)$. In the cases of $d=1,2$, the corresponding modular curves for each $G\in \Phi(d)$ have infinitely many points over the rationals and over quadratic fields, respectively. Therefore, $\Phi^{\infty}(d)=\Phi(d)$ for $d=1,2$.

The characterization of  $\Phi(d)$ for $d \geq 3$ is still open. Nevertheless, the uniform boundedness theorem, proved by Merel (and made effective by Oesterl\'e,  and later by Parent \cite{Parent2}),  states that there exists a constant $B(d)$, that only depends on $d$, such that $|G| \leq B(d)$, for all $G \in \Phi(d)$. Therefore, for a given $d$, only finitely many groups can appear as torsion subgroups of elliptic curves over a number field of degree $d$, that is, $\Phi(d)$ is finite for all $d\geq 1$. 
 For the case of cubic fields ($d=3$) there is recent  progress \cite{BN15,W15} to compute $\Phi(3)$.

The set $S(d)$ is slightly better understood. Parent \cite{Parent} has obtained $S(3)$ and Derickx, Kamienny, Stein and Stoll announced \cite{Derickx} that they established the sets $S(d)$ for $d=4,5$. The set $\Phi^{\infty}(d)$ has been determined for $d=3,4$ by Jeon et al. \cite{JKP04,JKP06}, and for $d=5,6$ by Derickx and Sutherland \cite{drew}. In particular:
\begin{eqnarray*}
\Phi^{\infty}(3) & \!\!\!\!=\!\!\!\! &\left\{ \cC_n \; | \; n=1,\dots,16,18,20 \right\} \cup \left\{ \cC_2 \times \cC_{2m} \; | \; m=1\dots,7 \right\}, \\[2mm]
\Phi^{\infty}(4) &\!\!\!\!=\!\!\!\! &\left\{ \cC_n \; | \; n=1,\dots,18,20,21,22,24 \right\} \cup \left\{ \cC_2 \times \cC_{2m} \; | \; m=1\dots,9 \right\} \cup \\
& &  \qquad \left\{ \cC_3 \times \cC_{3m} \; | \; m=1,2,3 \right\} \cup  \left\{ \cC_4 \times \cC_{4m} \; | \; m=1,2 \right\} \cup  \left\{ \cC_5 \times \cC_{5} \right\}  \cup  \left\{ \cC_6 \times \cC_{6}  \right\}.
\end{eqnarray*}
Recently, Najman \cite{N15a} has found that the elliptic curve \texttt{162b1} (in Cremona's notation \cite{cremonaweb}) has torsion $\cC_{21}$ over the cubic field $\Q(\zeta_9)^+ $. Therefore, $\Phi^{\infty}(3)\subsetneq \Phi(3)$. However, we do not  know any such example for the case of quartic fields, so it is not known whether $\Phi(4)=\Phi^\infty(4)$. It is worth pointing out that the fact that $\Phi(d)$ is finite together with the definition of $\Phi^\infty(d)$ imply that there are only finitely many isomorphism classes of elliptic curves over cubic and quartic fields such that their torsion subgroups are not isomorphic to one in the set $\Phi^{\infty}(3)$ or $\Phi^{\infty}(4)$, respectively. This remark justifies the following definition:
\begin{itemize}
\item We define $J(d)\subset \overline{\Q}$ as the finite set defined by the following property: $j\in J(d)$ if and only if there exists a number field $K$ of degree $d$, and an elliptic curve $E/K$ with $j(E)=j$, such that $E(K)_\text{tors}$ is isomorphic to a group in $\Phi(d)$ that is not in $\Phi^\infty(d)$. We denote by $J_\Q(d)\subset \Q$ the subset of $J(d)$ where we restrict to the case of elliptic curves $E$ defined over $\Q$. 
\end{itemize}
Since $\Phi(d)=\Phi^\infty(d)$ for $d=1,2$, it follows that $J(1)=J(2)=\emptyset$. Najman's example shows that $-140625/8\in J_\Q(3)$ .

In the case of elliptic curves with complex multiplication, we denote  $\Phi^{\cm}(d)$ the analogue of the set $\Phi(d)$ but restricting to elliptic curves with complex multiplication. The set $\Phi^{\cm}(1)$ was determined by Olson \cite{Olson74}; the quadratic and cubic cases by Zimmer et al. \cite{MSZ89,FSWZ90,PWZ97}; and recently, Clark et al. \cite{Clark2014} have computed the sets $\Phi^{\cm}(d)$, for $4\le d\le 13$. 

In this paper we are interested in the question of how the torsion subgroup of an elliptic curve grows  when we enlarge the field of definition. In particular, we consider elliptic curves defined over $\Q$, base extend to a quartic field, and study the growth in their torsion subgroup. For our purposes let us define the following sets:
\begin{itemize}
\item Let $\Phi_\Q(d)$  be the subset of $\Phi(d)$  such that $H\in \Phi_\Q(d)$ if there is an elliptic curve $E/\Q$ and a number field $K$ of degree $d$ such that $E(K)_\text{tors}\simeq H$. We define $S_\Q(d)\subseteq S(d)$, and $\Phi_\Q^\infty(d)\subseteq \Phi^\infty(d)$,  similarly.
\item Let $\Phi_{\Q}^{\star}(d)$ be the intersection of the sets $\Phi_\Q(d)$ and  $\Phi^{\infty}(d)$. 
\item For each $G \in \Phi(1)$, let $\Phi_\Q(d,G)$ be the subset of $\Phi_\Q(d)$ such that $E$ runs through all elliptic curves over $\Q$ such that $E(\Q)_{\tors}\simeq G$. Also, let $\Phi^{\star}_\Q(d,G)=\Phi_\Q(d,G)\cap \Phi^{\infty}(d)$.
\end{itemize}
\begin{remark}
It is important to notice that, a priori, $\Phi_\Q^\infty(d)\subseteq \Phi_\Q^\star(d)=\Phi_\Q(d)\cap \Phi^\infty(d)$ can be distinct sets. The set $\Phi_\Q^\infty(d)$ characterizes those torsion structures that appear infinitely often for elliptic curves defined over $\Q$, base extended to a degree $d$ number field. However, $\Phi_\Q^\star(d)$ characterizes torsion structures that occur infinitely often for elliptic curves defined over a degree $d$ number field, and also occur for elliptic curves defined over $\Q$ and base extended to some degree $d$ number field, but {\it perhaps} only for finitely many $\Q$-rational $j$-invariants. As we shall prove in Theorem \ref{main1}, we have $\Phi_\Q^\infty(4) \subsetneq \Phi_\Q^\ast(4)$ because $\cC_{15}\in \Phi^\infty(4)$ and $\cC_{15}\in \Phi_\Q(4)$, but $\cC_{15}$ does not belong to $\Phi_\Q^\infty(4)$, i.e., there are only finitely many $\overline{\Q}$-isomorphism classes of elliptic curves $E/\Q$ such that there is a quartic field $K/\Q$ with $E(K)_\text{tors}\simeq \cC_{15}$. 
\end{remark}

Let us review what is known for $S_\Q(d)$ and $\Phi_\Q(d)$. For $d\le 4$ we have:
$$
S_{\mathbb Q}(1)=S_{\mathbb Q}(2)=\{2,3,5,7\}\,\,\mbox{and}\,\, S_{\mathbb Q}(3)=S_{\mathbb Q}(4)=\{2,3,5,7,13\}.
$$
Moreover, the set $S_{\mathbb Q}(d)$ has been determined for $d \le 42$ by the second author \cite{Lozano}, together with a  conjectural description  for all $d\geq 1$ that holds if Serre's uniformity questions is answered positively. 
The sets $\Phi_{\mathbb Q}(d)$ have been completely described by Najman \cite{N15a} for $d = 2, 3$:
\begin{eqnarray*}
\Phi_{\mathbb Q}(2) &=& \left\{ \cC_n \; | \; n=1,\dots,10,12,15,16 \right\} \cup \left\{ \cC_2 \times \cC_{2m} \; | \; m=1,\dots,6 \right\} \cup \\
&& \qquad \left\{ \cC_3 \times \cC_{3r} \; | \; r=1,2 \right\} \cup \left\{ \cC_4 \times \cC_4 \right\},  \\[2mm]
\Phi_{\mathbb Q}(3) &=& \left\{ \cC_n \; | \; n=1,\dots,10,12,13,14,18,21 \right\} \cup \left\{ \cC_2 \times \cC_{2m} \; | \; m=1,\dots,4,7 \right\}.
\end{eqnarray*}

Chou \cite{chou} has completed a first step to determine  $\Phi_{\mathbb Q}(4)$, by classifying the possible torsion structures that may occur over Galois quartic fields.\footnote{After this article was submitted, the sets $\Phi_\Q(d)$ have been determined for $d=4$ (using results from this paper) by the first author and Najman \cite{GJN}, for $d=5$ by the first author \cite{enrique}, for $d=7$, and if $d$ is only divisible by primes $>7$ by the first author and Najman \cite{GJN}.} Moreover, Chou splits this classification depending on the Galois group of the quartic field. Let us denote by $\Phi^{\operatorname{V}_4}_\Q(4)$ (resp. $\Phi^{\cC_4}_\Q(4)$) when the quartic field has Galois group isomorphic to the Klein group $\operatorname{V}_4$ (resp. the cyclic group of order four, $\cC_4$). Then \cite[Theorem 1.3 and 1.4 ]{chou} shows that
\begin{eqnarray*}
\Phi^{\operatorname{V}_4}_\Q(4)  \!\!\!& = &\!\!\!\!\!\left\{ \cC_n \; | \; n=1,\dots,10,12,15,16 \right\} \cup \left\{ \cC_2 \times \cC_{2m} \; | \; m=1,\dots,6,8 \right\} \cup \\
& &  \left\{ \cC_3 \times \cC_{3m} \; | \; m=1,2 \right\} \cup  \left\{ \cC_4 \times \cC_{4m} \; | \; m=1,2 \right\} \cup  \left\{ \cC_6 \times \cC_{6}  \right\},\\[2mm]
\Phi^{\cC_4}_\Q(4) \!\!\!& = &\!\!\!\!\!\left\{ \cC_n \; | \; n=1,\dots,10,12,13,15,16 \right\} \cup \left\{ \cC_2 \times \cC_{2m} \; | \; m=1,\dots,6,8 \right\} \cup \left\{ \cC_5 \times \cC_{5}\right\}.
\end{eqnarray*}

Our first result determines $\Phi^\star_\Q(4)$ and $\Phi_\Q^\infty(4)$.

\begin{theorem}\label{main1} The set $\Phi^\star_\Q(4)$ is given by
\begin{eqnarray*}
\Phi_{\mathbb Q}^{\star}(4)\!\!\!& = &\!\!\!\!\!\left\{ \cC_n \; | \; n=1,\dots,10,12,13,15,16,20,24 \right\} \cup \left\{ \cC_2 \times \cC_{2m} \; | \; m=1,\dots,6,8 \right\} \cup \\
& &  \left\{ \cC_3 \times \cC_{3m} \; | \; m=1,2 \right\} \cup  \left\{ \cC_4 \times \cC_{4m} \; | \; m=1,2 \right\} \cup  \left\{ \cC_5 \times \cC_{5}  \right\} \cup  \left\{ \cC_6 \times \cC_{6}  \right\},
\end{eqnarray*}
and $\Phi_\Q^\infty(4)=\Phi_\Q^\star(4)\setminus \{\cC_{15} \}$. In particular, if $E/\Q$ is an elliptic curve with $j(E)\notin J_\Q(4)$ (a finite subset of $\Q$ defined above), then $E(K)_{\tors}\in \Phi_\Q^\star(4)$, for any number field $K/\Q$ of degree $4$. Moreover, if $E/\Q$ is an elliptic curve with $E(K)_\text{tors}\simeq \cC_{15}$  over some quartic field $K$, then $j(E) \in \{-5^2/2,\ -5^2\cdot 241^3/2^3, -5\cdot 29^3/2^5, 5\cdot 211^3/2^{15} \}$.
\end{theorem}

The set $\Phi_\Q(d)$ can be studied in more detail by analyzing the sets $\Phi_\Q(d,G)$ for each $G\in\Phi(1)$. Indeed, note that by definition we have
$\Phi_\Q(d) = \bigcup_{G\in \Phi(1)} \Phi_\Q(d,G).$ 
The sets $\Phi_{\mathbb Q}(d,G)$ have been calculated for $d=2,3$ in \cite{K97, GJT14,GJNT15}. Our second result determines $\Phi_\Q^\star(4,G)$ for each $G\in\Phi(1)$.

\begin{theorem}\label{main2}
For each $G \in \Phi(1)$, the set $\Phi^{\star}_\Q(4,G)$ is given in the following table:
$$
\begin{array}{|c|c|}
\hline
G & \Phi^{\star}_\mathbb{Q} \left(4,G \right)\\
\hline
\cC_1 & \left\{ \cC_1\,,\,\cC_{3}\,,\,   \cC_{5} \,,\,{ \cC_{7}} \,,\,{ \cC_{9}} \,,\,\cC_{13}\,,\, \cC_{15}\,,\, \cC_{3}\times\cC_{3}\,,\, \cC_{5}\times\cC_{5}\, \right\} \\
\hline
\cC_2 &\begin{array}{c} \left\{ \cC_2\,,
\cC_{4}\,,\,
\cC_{6}\,,\,
\cC_8\,,\,
\cC_{10}\,,\,
\cC_{12}\,,\,
\cC_{16}\,,\,
\cC_{20}\,,\,
\cC_{24}\,,\,
 {  \cC_2 \times \cC_{2}}\,,\, \right.\\ \left.
  \cC_2 \times \cC_{4}\,,\,
  \cC_2 \times \cC_{6}\,,\,
  \cC_2 \times \cC_{8}\,,\, 
  \cC_2 \times \cC_{10}\,,\,
  \cC_2 \times \cC_{12}\,,\, \right.\\ \left.
  \cC_2 \times \cC_{16}\,,\,
  \cC_3 \times \cC_{6}\,,\,
  \cC_4 \times \cC_{4}\,,\,
  \cC_4 \times \cC_{8}\,,\,
   \cC_6 \times \cC_{6}\, \right\} 
   \end{array}\\ 
\hline
\cC_3 & \left\{ \cC_3\,,\, \cC_{15}\,,\,{  \cC_{3}\times\cC_3}\,\right\} \\
\hline
\cC_4 & \begin{array}{c}
\left\{ \cC_4 \,,\,  
 \cC_{8}\,,\, 
 {  \cC_{12}}\,,\,
 \cC_{16}\,,\,
\cC_{24}\,,\,
{   \cC_2 \times \cC_{4}}\,,\,
  \cC_2 \times \cC_{8}\,,\, \right.\\ \left.
  \cC_2 \times \cC_{12}\,,\,
  \cC_2 \times \cC_{16}\,,\,
  \cC_4 \times \cC_{4}\,,\,
  \cC_4 \times \cC_{8}\, \right\} 
  \end{array}\\
\hline
\cC_5 & \left\{ \cC_5\,,\, {  \cC_{15} }\,,\,\cC_5\times\cC_5\, \right\} \\
\hline
\cC_6 & \left\{ \cC_6\,,\, 
\cC_{12}\,,\,
\cC_{24}\,,\,
  {    \cC_2 \times \cC_{6}}\,,\,
  \cC_2 \times \cC_{12}\,,\,
   {   \cC_3 \times \cC_{6} }\,,\,
  \cC_6 \times \cC_{6} \, \right\} \\
\hline
\cC_7 & \left\{ \cC_7  \,\right\} \\
\hline
\cC_8 & \left\{ \cC_8\,,\,\cC_{16},\,
 {   \cC_2 \times \cC_{8}}\,,\,
  \cC_2 \times \cC_{16}\,,\,
   \cC_4 \times \cC_8\, \right\} \\
\hline
\cC_9 & \left\{ \cC_9 \, \right\} \\
\hline
\cC_{10} & \left\{ \cC_{10}\,,\, 
\cC_{20}\,,\,
 {   \cC_2 \times \cC_{10}}\, \right\} \\
\hline
\cC_{12} & \left\{ \cC_{12}\,,\, \cC_{24}\,,\,
 {   \cC_2 \times \cC_{12}} \right\} \\
\hline
{\cC_2 \times \cC_2}& 
\begin{array}{c} \left\{ \cC_2 \times \cC_{2}\,,\,
 \cC_2 \times \cC_4\,,\,
 {   \cC_2 \times \cC_{6}} \,,\,
  \cC_2 \times \cC_8\,,\,  \right.\\ \left.
   \cC_2 \times \cC_{12}\,,\,
 \cC_2 \times \cC_{16}\,,\,
  \cC_4 \times \cC_4\,,\,
   \cC_4 \times \cC_8\, \right\}
   \end{array} \\
\hline
\cC_2 \times \cC_4 & \left\{ \cC_2 \times \cC_4\,,\,
 \cC_2 \times \cC_8\,,\,
 \cC_2 \times \cC_{16}\,,\,
  \cC_4 \times \cC_4\,,\,
   \cC_4 \times \cC_8\, \right\} \\
\hline
\cC_2 \times \cC_6 & \left\{ \cC_2 \times \cC_6\,,\, \cC_2 \times \cC_{12}\, \right\} \\
\hline
\cC_2 \times \cC_8 & \left\{ \cC_2 \times \cC_8\,,\,\cC_2\times  \cC_{16}\,,\,\cC_4\times \cC_8\, \right\} \\
\hline
\end{array}
$$
Further, for each $G\in\Phi(1)$, there is a finite set $J_\Q(4,G)\subset \Q$ of $j$-invariants such that if $E/\Q$ is an elliptic curve with $E(\Q)_\text{tors}\simeq G$ and $j(E)\notin J_\Q(4,G)$, then $E(K)_\text{tors}\in \Phi_\Q^\star(4,G)$, for any number field $K/\Q$ of degree $4$.
\end{theorem}

The finite sets $J_\Q(4)$ and $J_\Q(4,G)$ satisfy $J_\Q(4) = \bigcup_{G\in\Phi(1)} J_\Q(4,G)$. We finish the introduction with the following remark: if it turns out that $\Phi(4)=\Phi^\infty(4)$ (equivalently, $J(4)=\emptyset$) or if $J_\Q(4)=\emptyset$, then our results would determine $\Phi_\Q(4)$ and $\Phi_\Q(4,G)$ as well.

\begin{corollary}
If $\Phi(4)=\Phi^{\infty}(4)$ or $J_\Q(4)=\emptyset$, then $\Phi_{\Q}(4)=\Phi_{\mathbb Q}^{\star}(4)$ and $\Phi_{\Q}(4,G)=\Phi_{\mathbb Q}^{\star}(4,G)$ for any $G \in \Phi(1)$.
\end{corollary}

\begin{ack}
The authors would like to thank the referees for their comments and suggestions
\end{ack}

\section{Auxiliary results.}

We will use the Antwerp--Cremona tables and labels \cite{antwerp,cremonaweb} when referring to specific elliptic curves over $\Q$. The $\Q$-rational points on the modular curves $X_0(N)$ or, equivalently, the cyclic $\Q$-rational isogenies of elliptic curves over $\Q$, have been described completely in the literature, for all $N\geq 1$. One of the most important milestones in their classification was \cite{Mazur1978}, where Mazur dealt with the case when $N$ is prime. The complete classification of $\Q$-rational points on $X_0(N)$, for any $N$, was completed due to work of Fricke, Kenku, Klein, Kubert, Ligozat, Mazur  and Ogg, among others (see \cite[eq. (80)]{elkies1}; \cite{elkies2}; \cite{kleinfricke}, \cite[pp. 370-458]{fricke}; \cite[p. 1889]{ishii}; \cite{maier}; \cite{antwerp}; \cite{Mazur1978}; \cite{kenku}; or the summary tables in \cite{Lozano}).

\begin{theorem}\label{thm-ratnoncusps}\label{isogQ} Let $N\geq 2$ be a number such that $X_0(N)$ has a non-cuspidal $\Q$-rational point or, equivalently, let $E/\Q$ be an elliptic curve with a cyclic $\Q$-rational isogeny of degree $N$. Then:
\begin{enumerate}
\item $N\leq 10$, or $N= 12,13, 16,18$ or $25$. In this case $X_0(N)$ is a curve of genus $0$ and its $\Q$-rational points form an infinite $1$-parameter family, or
\item $N=11,14,15,17,19,21$, or $27$. In this case $X_0(N)$ is a curve of genus $1$, i.e.,~$X_0(N)$ is an elliptic curve over $\Q$, but in all cases the Mordell-Weil group $X_0(N)(\Q)$ is finite, or 

\item $N=37,43,67$ or $163$. In this case $X_0(N)$ is a curve of genus $\geq 2$ and (by Faltings' theorem) there are only finitely many $\Q$-rational points.
\end{enumerate} 
\end{theorem}

Table \ref{tableCM} lists the relevant cases of the sets $\Phi^{\cm}(d)$ (i.e., $d\le 7$ \cite{Olson74,MSZ89,FSWZ90,PWZ97,Clark2014}) that we will use in this article.

\begin{table}[h]
\begin{tabular}{|c|c|}
\hline
$d$& $\Phi^{\cm}(d)$ \\
\hline
  $1$ & $\left\{ \cC_1\,,\,  \cC_2\,,\,  \cC_3\,,\,  \cC_4\,,\,  \cC_6\,,\, \cC_2\times\cC_2\right\}$\\
   \hline
 $2$ & $\Phi^{\cm}(1)\cup \left\{ \cC_7\,,\,\cC_{10}\,,\,\cC_2\times\cC_4\,,\,\cC_2\times\cC_6\,,\,\cC_3\times\cC_3\,\right\}$\\
   \hline
 $3$ & $\Phi^{\cm}(1)\cup  \left\{\,\cC_9\,,\,\cC_{14}\,\right\}$\\
 \hline

 $4$ & $\Phi^{\cm}(2)\cup
 \left\{ \,\cC_5\,,\, \cC_8\,,\,\cC_{12}\,,\,\cC_{13}\,,\,\cC_{21}\,,\,\cC_2\times\cC_8\,,\,\cC_2\times\cC_{10}\,,\,\cC_3\times\cC_6\,,\,\cC_4\times\cC_4\,\right\}$\\
     \hline
  $5$ & $\Phi^{\cm}(1)\cup \left\{\, \cC_{11}\,\right\}$\\
       \hline
           
$6$ & $\Phi^{\cm}(2)\cup\Phi^{\cm}(3)\cup
 \left\{ \,\cC_{18}\,,\,\cC_{19}\,,\,\cC_{26}\,,\,\cC_2\times\cC_{14}\,,\,\cC_3\times\cC_6\,,\,\cC_3\times\cC_9\,,\, \cC_6\times\cC_6\,\right\}$\\
 \hline
    $7$ & $\Phi^{\cm}(1)$\\
       \hline      
\end{tabular}
\caption{$\Phi^{\cm}(d)$, for $d\le 7$.}\label{tableCM}
\end{table}

Let $E/\Q$ be a non-CM elliptic curve. For each prime $p$, let $\rho_{E,p}$ be the mod-$p$ Galois representation that describes the action of $\Gal(\Qbar/\Q)$ on the $p$-torsion $E[p]\simeq \Z/p\Z \oplus \Z/p\Z$ of $E$. Sutherland \cite{Sutherland2} and Zywina \cite{zywina1} have described all known (and conjecturally all) proper subgroups of $\GL(2, \Z/p\Z)$ that occur as the image of $\rho_{E,p}$ up to conjugacy. In particular, Sutherland \cite{Sutherland2} gives for each  $G_p=\rho_{E,p}(\Gal(\Qbar/\Q))\subsetneq \GL(2, \Z/p\Z)$ the following data:
\begin{itemize}
\item[$d_0$:] the index of the largest subgroup of $G_p$ that fixes a linear subspace of $E[p]$; equivalently, the degree of the minimal extension $L/\Q$ over which $E$ admits a $L$-rational $p$-isogeny.
\item[$d_1$:] is the index of the largest subgroup of $G_p$ that fixes a non-zero vector in $E[p]$;  equivalently, the degree of the minimal extension $L/\Q$ over which $E$ has a $L$-rational point of order $p$.
\item[$d$:] is the order of $G_p$; equivalently, the degree of the minimal extension $L/\Q$ for which $E[p]\subseteq E(L)$.
\end{itemize}
Table \ref{tableS35} is extracted from Table 3 of \cite{Sutherland2}, and it lists the values $d_0$, $d_1$, and $d$ for $p=3$, and $5$, for each possible image group $G_p\subseteq\ \GL(2, \Z/p\Z)$, where the groups are labeled as in \cite[\S 6.4]{Sutherland2}.

\begin{table}[h]
\begin{small}

\begin{tabular}{ccccc}

\begin{tabular}{crrr}
$G_3$&$d_0$&$d_1$&$d$\\
\toprule
\texttt{3Cs.1.1} & 1 & 1 & 2 \\ 
\texttt{3Cs} & 1 & 2 & 4 \\ 
\texttt{3B.1.1}  & 1 & 1 & 6 \\ 
\texttt{3B.1.2}& 1 & 2 & 6 \\ 
\texttt{3Ns} & 2 & 4 & 8 \\ 
\texttt{3B} & 1 & 2 & 12 \\ 
\texttt{3Nn} & 4 & 8 & 16 \\ 
\bottomrule
\end{tabular}

& \qquad

\begin{tabular}{crrr}
$G_5$&$d_0$&$d_1$&$d$\\
\toprule
\texttt{5Cs.1.1} & 1 & 1 & 4 \\ 
\texttt{5Cs.1.3} & 1 & 2 & 4 \\ 
\texttt{5Cs.4.1} & 1 & 2 & 8 \\ 
\texttt{5Ns.2.1} & 2 & 8 & 16 \\ 
\texttt{5Cs} & 1 & 4 & 16 \\ 
\texttt{5B.1.1} & 1 & 1 & 20 \\ 
\texttt{5B.1.2} & 1 & 4 & 20 \\ 
\bottomrule
\end{tabular}

&\qquad 

\begin{tabular}{crrr}
$G_5$&$d_0$&$d_1$&$d$\\
\toprule

\texttt{5B.1.4} & 1 & 2 & 20 \\ 
\texttt{5B.1.3} & 1 & 4 & 20 \\ 
\texttt{5Ns} & 2 & 8 & 32 \\ 
\texttt{5B.4.1} & 1 & 2 & 40 \\ 
\texttt{5B.4.2} & 1 & 4 & 40 \\ 
\texttt{5Nn} & 6 & 24 & 48 \\ 
\texttt{5B} & 1 & 4 & 80 \\ 
\texttt{5S4} & 6 & 24 & 96 \\ 
\bottomrule
\end{tabular}

\end{tabular}
\end{small}
\bigskip
\smallskip
\caption{Image groups $G_p=\rho_{E,p}(\Gal(\Qbar/\Q))$, for $p=3$, $5$, for non-CM elliptic curves $E/\Q$.}\label{tableS35}
\end{table}

In addition to Chou's classification of $\Phi_\Q^{V_4}(4)$ and $\Phi_\Q^{\cC_4}(4)$ already described in the introduction, we shall make use of the following result.

\begin{proposition}[\cite{chou}, Prop. 3.8]\label{chou} Let $p\equiv 3\mod 4$ be a prime with $p\geq 7$. Let $E/\Q$ be an elliptic curve and let $K/\Q$ be a quartic field such that $E(K)_\text{tors}$ contains a point $P$ of order $p$. Then, either:
\begin{itemize} 
\item $P$ is defined over $\Q$, i.e., $P\in E(\Q)[p]$, or
\item There is a subfield $F\subseteq K$, $[F:\Q]=2$ such that $P\in E(F)[p]$.
\end{itemize} 
\end{proposition}

\newpage

We will also quote the following result of Najman.

\begin{proposition}[\cite{N15a}, Lemma 5]\label{najman} Let $F$ be a quadratic field, $n$ an odd positive integer, and $E/\Q$ an elliptic curve such that $E(F)$ contains $\cC_n$. Then $E/\Q$ has an
$n$-isogeny.
\end{proposition}

The determination of $\Phi_\Q^\infty(4)$, $\Phi^{\star}_\Q(4)$ and $\Phi^{\star}_\Q(4,G)$ will rest on the following result:

\begin{theorem}\label{teo}
Let $E/\Q$ be an elliptic curve and $K/\Q$ a quartic number field such that $E(\Q)_{\tors}\simeq G$ and $E(K)_{\tors}\simeq H$.
\begin{enumerate}
\item\label{t1} If $\cC_2 \not\subset G$, then $\cC_2 \not\subset H$.
 \item\label{t2}$11$ and $17$ do not divide the order of $H$.
\item\label{t3} $\cC_{14}$, $\cC_{2}\times\cC_{14}\notin \Phi_\Q(4)$.
\item\label{t21}$\cC_{21}\not\subset H$.
\item\label{t4} If $\cC_4\subseteq G$, then $\cC_{20}\not\subset H$.
\item\label{t5} If $\cC_8\subseteq G$, then $\cC_{24}\not\subset H$.
\item \label{t6} If $\cC_2\times\cC_2\subseteq G$, then $\cC_2\times\cC_{10}\not\subset H$.
\item\label{t7} If $\cC_2\times\cC_4\subseteq G$, then $\cC_2\times\cC_{12}\not\subset H$.

 \item\label{t8}If $H= \cC_{6}\times\cC_6$, then $G=\cC_2$ or $G=\cC_6$. 
\item\label{t9} If $P\in E(K)[9]$, then there exists a subfield $F\subsetneq K$ such that $P\in E(F)[9]$.

\item\label{t10} If $\cC_{18}$, $\cC_3 \times \cC_{9}\not\subset H$.
\item\label{t11} If $G= \cC_3$, then $\cC_{9}\not\subset H$.
\end{enumerate}
\end{theorem}

\begin{remark}
If in the statements  (\ref{t4})--(\ref{t7}) the quartic field $K$ is replaced by a number field such that $4 \le [K:\Q]\le 7$, then those statements still hold true. The reason is that in the proofs of these statements in the non--CM cases we use that $d_1\le 4$, but in fact in all those cases $d_1\le 7$ (here $d_1$ is the quantity associated to $\rho_{E,p}$ that appears in Table \ref{tableS35}). For the CM case, in the proof of (\ref{t4}) (resp. (\ref{t5}), (\ref{t7})) we used that $\cC_{20}\subseteq H$ (resp. $\cC_{24}$,  $\cC_2\times\cC_{12}$) is not a subgroup of one of the groups in $\Phi^{\cm}(4)$,  but the same is true for $d\le 7$ (see Table \ref{tableCM}).
\end{remark}

\begin{proof}
(\ref{t1}) If $E$ has a short Weierstrass equation of the form $y^2=f(x)$, where $f(x)\in\Z[x]$ is a monic cubic polynomial, the hypothesis $\cC_2 \not\subset G$ implies the irreducibility of $f(x)$ over $\Q$, hence over $K$.\\

(\ref{t2}) By \cite{Lozano} we have $11,17\notin S_{\Q}(4)=\{2,3,5,7,13\}$.\\

(\ref{t3}) By \cite[Prop. 3.9]{chou} we have $\cC_{14},\cC_{2}\times\cC_{14}\notin \Phi_\Q(4)$. \\

(\ref{t21}) Suppose that $\langle P \rangle\oplus\langle Q\rangle\subseteq E(K)_{\tors}$, where $P$ and $Q$ are points of order $3$ and $7$ respectively, and let $F_3=\Q(P)$ and $F_7=\Q(Q)$ be the fields of definition of each point. By Prop. \ref{chou}, the field $F_7\subset K$ is at most quadratic. Since $\cC_{21}$ is not a subgroup of one of the groups in $\Phi_\Q(1)\cup\Phi_\Q(2)$, it follows that $P$ is not defined over $F_7$, and if $F_7=\Q$, then $F_3$ cannot be quadratic. If $F_3$ was quadratic, then $F_7$ would be quadratic also with $F_3\neq F_7$, and so $K$ would be a biquadratic field. But Chou's classification of $\Phi_\Q^{V_4}(4)$ (see our introduction) shows that $\cC_{21}$  is not a subgroup of one of the groups in $\Phi_\Q^{V_4}(4)$.  It follows that $F_3$ must be a quartic, and so $K=F_3$. Finally, notice that the same argument shows that if $R\in E[3]$ is any other non-trivial point of order $3$, then $[\Q(R):\Q]\geq 3$. Hence, if $d_1$ is the quantity associated to the image of $\rho_{E,3}$  in the notation of \cite{Sutherland2}, then $3\leq d_1\leq 4$. We consider two cases depending on whether $E$ has CM.

Let $E/\Q$ be without CM. Since $3\leq d_1\leq 4$, looking at Table \ref{tableS35} we conclude that the image of the Galois representation $\rho_{E,3}$ must be isomorphic to \texttt{3Ns} ($G_2$ in Zywina notation \cite[\S1.2]{zywina1}), that is, a normalizer of split Cartan. Zywina \cite[Theorem 1.2]{zywina1} has determined the $j$-invariant of elliptic curves with \texttt{3Ns} image:
$$
J_2(t)=27\frac{(t+1)^3(t-3)^3}{t^3},\qquad\mbox{for some $t\in\Q$}.
$$  
On the other hand, $E/\Q$ has a $\Q$-rational $7$-isogeny since $\cC_7\subset E(F_7)$ and $[F_7:\Q]\leq 2$, by Proposition \ref{najman}. Then, we observe in \cite[Table 3]{Lozano} that  its $j$-invariant must be of the form:
$$
j_7(h)=\frac{(h^2 +13h+49)(h^2+5h+1)^3}{h}, \qquad\mbox{for some $h\in\Q$}.
$$
The above $j$-invariants should be equal, so $J_2(t)=j_7(h)$. In particular, since $J_2(t)$ is a cube,  we must have 
$$
hs^3=h^2+13h+49,\qquad\mbox{for some $h,s\in\Q$}.
$$
This equation defines a curve $C$ of genus $2$, which in fact transforms (according to \texttt{Magma} \cite{magma}) to $C'\,:\,y^2 = x^6-26x^3-27$.\footnote{A remarkable fact is that this genus $2$ curve is {\it new modular} of level $63$ (see \cite{EGJ}).} The jacobian of $C'$ has rank $0$, so we can use the Chabauty method, and determine that the points on $C'$ are 
$$
C'(\Q)=\{  (-1 , 0),(3,0)\}\cup\{ (1 : \pm 1: 0)\}.
$$
Therefore 
$$
C(\Q)=\{ (7,3),(-7,-1)\}\cup\{(0 : 1 : 0),(1:0:0)\}.
$$
Now, the corresponding $j$-invariants are $j=3^3\cdot 5^3\cdot 17^3$ and $j=-3^3\cdot 5^3$, that belong to CM elliptic curves. This finishes the proof in the non-CM case.

Now, suppose $E/\Q$ has CM. As seen above, $E/\Q$ must have a $\Q$-rational $7$-isogeny, and the only curves with CM and a $7$-isogeny have CM by $\Q(\sqrt{-7})$ (see for example Section 7.1, Table 1, of \cite{GJLR15}). Moreover, since $-7$ is a quadratic non-residue modulo $3$, it follows that the image of $\rho_{E,3}$ is \texttt{3Nn} by Theorem 7.6 of \cite{Lozano}. However, $d_1=8$ by Table  \ref{tableS35}, which contradicts the fact that $E/\Q$ has a point of order $3$ defined over $K=F_3$, a quartic field. Thus, there is no elliptic curve $E/\Q$ with CM and a $21$-torsion point defined over a quartic number field, which concludes the proof of part (\ref{t21}).\\

(\ref{t4}) Suppose for a contradiction that $\cC_4\subseteq G$ and $\cC_{20}\subseteq H$. $E$ has no CM since  $\cC_{20}$ is not a subgroup of one of the groups in $\Phi^{\cm}(4)$, by Table \ref{tableCM}. Moreover, there exists $P\in E(K)[5]$ not defined over $\Q$. That is, $d_1\leq 4$ for the image of $\rho_{E,5}$.  Looking at the Table \ref{tableS35} we check that in all the possible images with  $d_1\leq 4$ we have $d_0=1$. Therefore $E$ has a $\Q$-rational $5$-isogeny. Then, since $E$ has a point of order $4$ defined over $\Q$, there exists a $20$-isogeny defined over $\Q$, which contradicts Theorem \ref{isogQ}.\\

(\ref{t5}) Suppose for a contradiction that $\cC_8\subseteq G$  and $\cC_{24}\subseteq H$. As in case (\ref{t4}) we conclude that $E$ has no CM and $d_1\leq 4$ for the image of $\rho_{E,3}$. In this case, Table \ref{tableS35} shows that $d_0\in \{1,2\}$. If $d_0=1$, then there exists a $24$-isogeny defined over $\Q$, in contradiction with Theorem \ref{isogQ}. If $d_0=2$, then the image of the Galois representation $\rho_{E,3}$ is labelled \texttt{3Ns}. Similar to the proof of (\ref{t21}) the $j$-invariant of $E/\Q$ is a perfect cube. On the other hand, since $E/\Q$ has a point of order $8$ defined over $\Q$, the curve $E/\Q$ has a rational $8$-isogeny.   Looking at Table 3 in \cite{Lozano} we have that its $j$-invariant is of the form:
$$
j_8(h)=\frac{(h^4 - 16h^2 +16)^3}{(h^2-16)h^2}, \qquad\mbox{for some $h\in\Q$}.
$$
Then we must have $j_8(h)=s^3$ for some $s\in\Q$, and this gives us the equation: 
$$
(h^2-16)h^2 = s^3,\qquad\mbox{for some $h,s\in\Q$}.
$$
This equation defines a curve $C$ of genus $2$, which in fact transforms (according to \texttt{Magma} \cite{magma}) to $C'\,:\,y^2 = x^6+1$. The jacobian of $C'$ has rank $0$, so we can use the Chabauty method, and determine that the points on $C'$ are 
$$
C'(\Q)=\{  (0 , \pm 1)\}\cup \{(1 : \pm 1 : 0)\}.
$$
Therefore 
$$
C(\Q)=\{ (\pm 4,0),(0,0)\}\cup \{(0 : 1 : 0)\}.
$$
These are cusps in $X_0(8)$, and so we have reached a contradiction to the existence of such curve $E$. This finishes the proof.\\

(\ref{t6}) Suppose that $\cC_2\times\cC_2\subseteq G$ and $\cC_2\times\cC_{10}\subseteq H$. If $E$ has no CM, then we can conclude that $E/\Q$ has a $\Q$-rational $5$-isogeny as in the proof of case (\ref{t4}). However, since $\cC_2\times \cC_2 \subseteq G \simeq E(\Q)_\text{tors}$, then $E$ is $2$-isogenous to two curves $E'$ and $E''$, such that $E$, $E'$, and $E''$ are all non-isomorphic pairwise. It follows that there is a $\Q$-rational $4$-isogeny from $E'$ to $E''$ that is necessarily cyclic. Moreover, since $E$ has a $5$-isogeny, if follows that $E'$ also has a $\Q$-rational $5$-isogeny, and therefore $E'$ would have a $\Q$-rational $20$-isogeny which is impossible by Theorem \ref{isogQ}. 

If $E$ has CM, with $\cC_2\times \cC_2 \subseteq  E(\Q)_\text{tors}$, then by counting independent $\Q$-rational $2$-isogenies, we see that $j(E)=1728$ and $E$ has a Weierstrass model of the form $y^2=x^3-r^2x$, for some $r\in \Q$ (see \cite{GJLR15}, Section 7.1, Table 1). In particular, $E$ has CM by $\Q(i)$ and, since $-1$ is a square modulo $5$, the image of $\rho_{E,5}$ must be isomorphic to \texttt{5Ns} (that is, the normalizer of split Cartan) by Theorem 7.6 of \cite{Lozano}. However, Table \ref{tableS35} shows that $d_1=8$ for such image, i.e., a point of order $5$ is defined in an extension of degree $\geq 8$, which contradicts the fact that there is a point defined over $K$, an extension of degree $4$. This finishes the proof.\\

(\ref{t7}) Suppose that $\cC_2\times\cC_4\subseteq G$ and $\cC_2\times\cC_{12}\subseteq H$. We first note that $E$ does not have CM because $\cC_2\times\cC_{12}$ is not a subgroup of one of the groups in $\Phi^{\cm}(4)$, by Table \ref{tableCM}. As in case (\ref{t5}), we have $d_1\leq 4$ for the image of $\rho_{E,3}$ and Table \ref{tableS35} shows that $d_0\in \{1,2\}$. Moreover, the case $d_0=1$ is not possible because $E$ is $2$-isogenous to a curve $E'$ that would have a $\Q$-rational $24$-isogeny, which do not exist by Theorem \ref{isogQ}. If $d_0=2$, then the image of $\rho_{E,3}$ is \texttt{3Ns} and, as pointed out in case (\ref{t21}), this implies that $E$ has $j$-invariant $J_2(t)$ for some $t\in \Q$. Therefore $E$ is $\overline{\Q}$-isomorphic to the elliptic curve
$$
E'_t\,:\,y^2 + xy = x^3 - \frac{36}{J_2(t) - 1728}x - \frac{1}{J_2(t) - 1728}.
$$
In particular, $E$ and $E'_t$ are quadratic twists of each other, and their discriminants satisfy $\Delta(E)=u^{6}\Delta(E'_t)$, for some non-zero $u\in\Q$. On the other hand, since the full $2$-torsion is defined over $\Q$ we have that $\Delta(E)$ is a square (and hence so is $\Delta(E_t')$). That is:
$$
3t(t^2-6t-3)=r^2, \qquad\mbox{for some $r\in\Q$}.
$$ 
This equation defines an elliptic curve (\texttt{36a4}) with only two rational points, namely $(r,t)=(0,0)$ and $(1:0:0)$. These points do not correspond to elliptic curves. This finishes the proof.\\

(\ref{t8}) Suppose that $H= \cC_{6}\times\cC_6$. By Table \ref{tableCM}, the curve $E/\Q$ cannot have CM, so let us assume that $E$ is not CM. Since $\cC_{3}\times\cC_3 \subseteq H$, we have that $d|4$ for the image of $\rho_{E,3}$. Looking at the Table \ref{tableS35} we check that $d=2$ (\texttt{3Cs.1.1}) or $d=4$ (\texttt{3Cs}), so we treat each case separately.

For the case \texttt{3Cs.1.1} we have $d_1=1$, that is $\cC_3\subseteq G$. Now, since $|H|$ is even, it follows that $|G|$ must be even by (\ref{t1}), and so $\cC_6\subseteq G$. On the other hand, since $d=2$ for \texttt{3Cs.1.1}, there exists a quadratic field $F\subset K$ such that $\cC_3\times\cC_3 \subseteq E(F)_{\tors}$. Then \cite[Theorem 2]{GJT15} shows that $G=\cC_6$.

Now suppose that the image of $\rho_{E,3}$ is \texttt{3Cs}. We have $d_1=2$, therefore $\cC_3\not \subseteq G$. As before, $|G|$ is even (since $|H|$ is even) and $G\subseteq H$, then $G=\cC_2$ or $G=\cC_2\times\cC_2$. We are going to discard the latter case. 
Zywina \cite[Theorem 1.2]{zywina1} has determined the $j$-invariant of curves with mod $3$ image conjugate to \texttt{3Cs} ($G_1$ in Zywina notation \cite[\S1.2]{zywina1}):
$$
J_1(t)=27\frac{(t+1)^3(t+3)^3(t^2+3)^3}{t^3(t^2+3t+3)^3},\qquad\mbox{for some $t\in\Q$}.
$$  
As in the case of (\ref{t7}), the fact that the full $2$-torsion is defined over $\Q$ implies that the discriminant of $E$ must be a square. This implies:
$$
3t(t^2+3t+3)=r^2, \qquad\mbox{for some $r\in\Q$}.
$$ 
This equation defines an elliptic curve (\texttt{36a3}) which has only the rational points $(0,0)$ and $(1:0:0)$, which do not correspond to elliptic curves. This finishes the proof.\\

(\ref{t9}) Let $P\in E(K)[9]$ be a point of order $9$ on $E/\Q$, with $[K:\Q]=4$. We shall assume that $\Q(P)=K$ because, otherwise, $\Q(P)$ is trivial or quadratic over $\Q$. In particular, this implies that $K\subseteq \Q(E[9])$. Consider $m=[K:K\cap \Q(E[3])]$. On one hand, we have that 
$$
m=[K:K\cap \Q(E[3])]=[K\Q(E[3]):\Q(E[3])],
$$
and, therefore, $m$ divides $[\Q(E[9]):\Q(E[3])]$, which is a power of $3$ (because $\Gal(\Q(E[9])/\Q(E[3]))\subseteq \GL(2,\Z/9\Z)/\GL(2,\Z/3\Z)\cong (\Z/3\Z)^4$). On the other hand, $m$ is a divisor of  $[K:\Q]=4$. It follows that $m=1$ and $K\subseteq \Q(E[3])$.

Since $\Q(E[3])/\Q$ is Galois, it follows that the Galois closure $\widehat{K}$ of $K$ in $\overline{\Q}$ is also contained in $\Q(E[3])$. Since $K\subseteq \widehat{K}$, we know that $E(\widehat{K})$ contains $P$. We distinguish three cases according to whether $E(\widehat{K})[9]$ is isomorphic to $\Z/9\Z$, $\Z/9\Z\times \Z/3\Z$, or $\Z/9\Z\times\Z/9\Z$, and we shall prove that all cases lead to a contradiction.
\begin{itemize}
\item $E(\widehat{K})[9]\cong \Z/9\Z$. Then, $\langle P \rangle$ is a Galois-stable subgroup of order $9$. In particular, the field of definition of $P$, that is, $K=\Q(P)$, is Galois and it is isomorphic to a subgroup of $(\Z/9\Z)^\times \cong \Z/6\Z$. Since $[K:\Q]=4$, this is impossible.
\item $E(\widehat{K})[9]\cong \Z/9\Z\times \Z/3\Z$. Since $\widehat{K}/\Q$ is Galois, this implies that $\langle 3P\rangle$ is a Galois-stable subgroup of order $3$. In particular, the Galois representation associated to $E[3]$ has an image isomorphic to an upper triangular subgroup $G$ of 
$$B=\left\{ \left(\begin{array}{cc}
a & b\\
0 & c\\
\end{array} \right)\right\} \subseteq \GL(2,\Z/3\Z).$$ Since $K\subseteq \widehat{K} \subseteq \Q(E[3])$, and $[K:\Q]=4$, and $|B|=4\cdot 3$, it follows that the subgroup $H$ of $G$ that fixes $K$ must be trivial or of order $3$. Since such a group $H$ is normal in $G$, as a consequence we obtain $K=\widehat{K}$ and $\Gal(K/\Q)\cong \Z/2\Z \times \Z/2\Z$. However, by Theorem 1.4 of \cite{chou}, it is impossible for a biquadratic extension $K$ to have a torsion subgroup $E(K)[9]\cong \Z/9\Z\times \Z/3\Z$.
\item $E(\widehat{K})[9]\cong \Z/9\Z\times \Z/9\Z$. Since $\widehat{K}\subseteq \Q(E[3])$, then this means that $\Q(E[9])=\Q(E[3])$. In particular, $\Q(\zeta_9)\subseteq \Q(E[3])$.  Let $G\subseteq \GL(2,\Z/3\Z)$ be the image of $\rho_{E,3}$. If $G\neq \GL(2,\Z/3\Z)$ and $E$ has no CM, then $G$ is one of the groups labelled \texttt{3Cs.1.1}, \texttt{3Cs}, \texttt{3B.1.1}, \texttt{3B.1.2}, \texttt{3Ns}, \texttt{3B}, or \texttt{3Nn} (see Table \ref{tableS35}). If $E$ has CM, then by Proposition 1.14 of \cite{zywina1}, $G$ is one of \texttt{3Ns}, \texttt{3Nn}, \texttt{G}, \texttt{H1}, or \texttt{H2}. However, none of these groups have a subgroup $H$ such that $G/H\cong \Gal(\Q(\zeta_9)/\Q)\cong \Z/6\Z$. It follows that $G=\GL(2,\Z/3\Z)$. 

Thus, $E/\Q$ is a curve such that $\rho_{E,3}$ is surjective, but $\rho_{E,9}$, the representation associated to $E[9]$ is {\it not} surjective. Moreover, the image of $\rho_{E,3}$ and $\rho_{E,9}$ are isomorphic (because $\Q(E[3])=\Q(E[9])$ in our case). However, the elliptic curves over $\Q$ such that $\rho_{E,3}$ is surjective but $\rho_{E,9}$ is not surjective where classified by Elkies \cite{elkies3} and for such curves $\Gal(\Q(E[3])/\Q)$ has size $48$, while $\Gal(\Q(E[9])/\Q)$ has size $144$. Therefore, this third possibility is also impossible in our setting.
\end{itemize}

\

(\ref{t10}) By (\ref{t9}) we know that if $P\in E(K)[9]$, then there exists a subfield $F\subsetneq K$ such that $P\in E(F)[9]$.
\begin{itemize}
\item Suppose that $\cC_3\times\cC_9\simeq \langle P \rangle\oplus\langle Q\rangle\subseteq E(K)_{\tors}$, where $P$ and $Q$ are points of order $3$ and $9$ respectively, and $Q$ is defined over the quadratic field $F\subset K$. The point $P+Q$ also has order $9$, and it is therefore defined over a quadratic field $F'\subset K$. If $F'=F$ then $\cC_3\times\cC_9\subseteq E(F)_{\tors}$. But $\cC_3\times\cC_9$ is not a subgroup of one of the groups in $\Phi_{\Q}(2)$. If $F'\ne F$, then $\cC_3\times\cC_9\subseteq E(FF')_{\tors}$. But $K=FF'$ is a biquadratic field and $\cC_3\times\cC_9$ is not a subgroup of one of the groups in $\Phi^{\operatorname{V}_4}_{\Q}(4)$.
\item Suppose that $\cC_{18}\subseteq H$. By (\ref{t1}) we have $G=\cC_2$, or $\cC_6$. Then $ \cC_{18}\subseteq E(F)_{\tors}$, but $\cC_{18}$ is not a subgroup of one of the groups in $\Phi_\Q(2)$.

\end{itemize}

(\ref{t11}) Suppose that $G= \cC_3$ and $\cC_{9}\subseteq H$. By (\ref{t9}) there exists a quadratic field $F\subset K$ such that $\cC_9 \subseteq E(F)_{\tors}$. But this is impossible by \cite[Theorem 2]{GJT15}.\\
\end{proof}

\begin{theorem}\label{c15}
Let $E/\Q$ be an elliptic curve. If $E(K)_\text{tors}\simeq \cC_{15}$  over some quartic field $K$, then $j(E) \in \{-5^2/2,\ -5^2\cdot 241^3/2^3,\ -5\cdot 29^3/2^5,\ 5\cdot 211^3/2^{15} \}$. Moreover, the field of definition of the torsion point of order $15$ is abelian over $\Q$.
\end{theorem}
\begin{proof}
Let $E/\Q$ and $K$ be as in the statement of the theorem, such that $E(K)_\text{tors}=\langle R \rangle$, where $R\in E$ is a point of exact order $15$. We will first show that $\Q(R)$ is abelian. 

Let $P_3=5R$ and $P_5 = 3R$. By Table \ref{tableCM}, we know that $\cC_{15}\notin \Phi^{\text{CM}}(4)$, so $E/\Q$ does not have CM. Let $G_5$ be the image of $\rho_{E,5}$. Since $R$ is defined over $K$, the point $P_5$ of order $5$ is defined in degree $1$, $2$, or $4$, and so $d_1(G_5) \le 4$. By Table \ref{tableS35}, the image $G_5$ is a subgroup of the Borel
$$\left\{\left(\begin{array}{cc}\ast & \ast\\
0 & \ast \end{array} \right)  \right\}\subset \GL(2,\Z/5\Z).$$
Since $P_5$ is defined over a quartic field $K$, it follows that $P_5$ is contained in the fixed field of the subgroup
$$G_5\cap  \left\{\left(\begin{array}{cc}1 & \ast\\
0 & 1 \end{array} \right)  \right\}\subset \GL(2,\Z/5\Z).$$
In particular, $\Q(P_5)$ is contained in a Galois extension  with Galois group $\subseteq (\Z/5\Z)^\times \oplus (\Z/5\Z)^\times$. It follows that $\Q(P_5)$ is Galois and abelian.

Similarly, consider $G_3$, the image of $\rho_{E,3}$. By Table \ref{tableS35}, either $G_3$ is a subgroup of the Borel of $\GL(2,\Z/3\Z)$, or $G_3$ is \texttt{3Ns}. If $G_3$ is contained in a Borel, then as in the case of $p=5$, we conclude that $\Q(P_3)$ is abelian, and therefore $\Q(P_3,P_5)=\Q(R)$ is abelian. Otherwise, suppose that the image of $\rho_{E,3}$ is \texttt{3Ns}. By \cite{zywina1}, the $j$-invariant of $E$ is of the form $j(E) = J_2(t)$ for some $t\in\Q$, with 
$$J_2(t) = 27\frac{(t+1)^3(t-3)^3}{t^3}.$$
On the other hand, we know that $G_5$ is contained in a Borel subgroup of $\GL(2,\Z/5\Z)$ and therefore $E/\Q$ has a $\Q$-rational $5$-isogeny. Using the tables of \cite{Lozano}, we see that $j(E) = j(h)$ for some $h\in\Q$, where 
$$j_5(h) = \frac{(h^2+10h+5)^3}{h}.$$
Thus, $j_5(h) = J_2(t)$. Since $J_2(t)$ is a perfect cube we must have $h=s^3$ and the pair $(s,t)$ is a point on
$$C: (s^6+10s^3+5)t = 3(t+1)(t-3)s.$$
The curve $C$ has genus $1$, and there is a degree $1$ rational map $\phi: C \to E'$, where $E'$ is the elliptic curve \texttt{15a3}. Now, the curve $E'$ has finite Mordell-Weil group, isomorphic to $\cC_2\times\cC_4$. The rational points 
$$
S=\{(-5/2 , 9/32), (-5/2 , -32/3), (-2 , -2/3), (0 , 0 ), (-2 , 9/2) \}\cup\{(1 : 0 : 0)\}
$$
on $C$ map to $6$ rational points on $E'$, while $(0:1:0)\in C$ is singular (a node) and once the singularity is resolved, the two points on the desingularization $\widehat{C}$ of $C$ map to the two remaining rational points on $E'$. It follows that $C(\Q)=S\cup\{(0:1:0)\}$. The non-cuspidal points on $C(\Q)$ correspond to the following $j$-invariants:
$$\{ 11^3/2^3,\ -29^3\cdot 41^3/2^{15} \}.$$
Examples of curves with $j$-invariants $11^3/2^3$ and $-29^3\cdot 41^3/2^{15}$, respectively, are  \texttt{338d1} and \texttt{338d2}. For both curves, $\Q(P_5)$ is a cyclic quartic, and by Lemma 9.6, part (3), of \cite{Lozano}, every curve with such $j$-invariants has the same property. It follows that $\Q(P_5)=\Q(R)=K$ is Galois, and abelian.

Therefore, we have shown that in all cases $\Q(R)$ is Galois,  abelian, and of degree dividing $4$. If so, then $E/\Q$ must have a $\Q$-rational $15$-isogeny. By \cite{Lozano}, Table 4, there are only $4$ possible $j$-invariants, namely,
$$\{-5^2/2,\ -5^2\cdot 241^3/2^3,\ -5\cdot 29^3/2^5,\ 5\cdot 211^3/2^{15}  \}.$$
Elliptic curves with these $j$-invariants that reach $\cC_{15}$ in a quartic extension are \texttt{50a1}, \texttt{450b2}, \texttt{50a3}, and \texttt{50a4}, respectively. This completes the proof of the theorem.
\end{proof} 

We will use the following result, known as the $2$-divisibility method.

\begin{theorem}[\cite{JKL13}, Theorem 3.1; \cite{GJT15}, Prop. 12]\label{2div}
Let $E$ be an elliptic curve over a number field $k$ with a $k$-rational $N$-torsion point $P$. Then $E$ has
a $K$-rational $2N$-torsion point $Q$, where $K$ is a quartic extension field of $k$. Moreover, the same result holds if we replace $k$ by $k(t)$, and $K/k(t)$ a quartic extension.
\end{theorem} 

Now we apply the $2$-divisibility method to the cases of $\cC_{20}$ and $\cC_{24}$.

\begin{theorem}\label{c20c24}
There are infinitely many non-isomorphic (over $\overline{\Q}$) elliptic curves $E/\Q$ such that there is a quartic field $K$ with $E(K)_\text{tors}\simeq \cC_{20}$ (resp. $\cC_{24}$).
\end{theorem}
\begin{proof}
Kubert \cite[Table 3]{Kubert} gave for each $G\in \Phi(1)$ a one-parameter family  
$$
{\mathcal T}^G_{t}:\ y^2+(1-c)xy-by=x^3-bx^2,\qquad \mbox{where $b,c\in\Q(t)$},
$$
such that ${\mathcal T}^G_{t}(\Q(t))_\text{tors}\simeq G$ and, in fact, for all but finitely many values of $t_0\in\Q$, the curve ${\mathcal T}^G_{t_0}/\Q$ has a subgroup $G$ in its torsion subgroup over $\Q$. When $G=\cC_{10}$ (resp.  $\cC_{12}$), Mazur's classification of the torsion subgroups that can occur over $\Q$ implies that ${\mathcal T}^G_{t_0}(\Q)_{\text{tors}}\simeq G$ for all but finitely many $t_0\in\Q$. The equation ${\mathcal T}^G_{t}$ is called the Kubert-Tate normal form. For the cases we are interested in, we have
$$
\begin{array}{lclcl}
  G=\cC_{10} & : &  c=(2t^3-3t^2+t)/(t-(t-1)^2) & , &  b=ct^2/(t-(t-1)^2),\\
  G=\cC_{12} & : &  c=(3t^2-3t+1)(t-2t^2)/(t-1)^3 & , & b=c(2t-2t^2-1)/(t-1),
\end{array}
$$
and the point $P=(0,0)$ has order $10$ and $12$ respectively. Now, we can use the $2$-divisibility method (Theorem \ref{2div}) to halve $P$. This method  allows to build an extension $L/\Q(t)$ of degree $4$ and a point $Q\in {\mathcal T}^G_{t}(L)$ such that $2Q=P$. As mentioned above, for all but finitely many $t_0\in \Q$ the curve 
${\mathcal T}^G_{t_0}/\Q$ satisfies ${\mathcal T}^G_{t_0}(\Q)_{\tors}\simeq \cC_{10}$ (or $\cC_{12}$, respectively).  Then, by the $2$-divisibility method we find a number field $L_{t_0}/\Q$ of degree dividing $4$ such that ${\mathcal T}^G_{t_0}(L_{t_0})_{\tors}\simeq  \cC_{20}$ (resp.  $\cC_{24}$). Since  $\cC_{20},\cC_{24}\notin \Phi_{\Q}(d)$ for $d\le 3$, we have that $[L_{t_0}:\Q]=4$ for any $t_0\in\Q$. Since the $j$-invariant of $T_{b,c}$ is not constant, this proves that there are infinitely many non $\overline{\Q}$-isomorphic elliptic curve over $\Q$ with torsion structures  $\cC_{20}$ and $\cC_{24}$ over quartic fields. 
\end{proof}

\begin{remark}
One can construct explicit infinite families of elliptic curves with the properties of Theorem \ref{c20c24} using the recipe described by Proposition 12 in \cite{GJT15}.
\end{remark}

\section{Proof of Theorems \ref{main1} and \ref{main2}}

We are ready to prove our main theorems. In Theorem \ref{main1}, we shall first determine the isomorphism classes that appear in $\Phi^\star_\Q(4)= \Phi_\Q(4)\cap \Phi^\infty(4)$ and then use that information to determine $\Phi_\Q^\infty(4)$. 

\begin{proof}[Proof of Theorem \ref{main1}]
Let $E/\Q$ be an elliptic curve, $K$ a quartic number field, $G\in \Phi_\Q(1)$ and $H\in \Phi^{\star}_\Q(4)$ such that $G\simeq E(\Q)_{\tors}\subseteq E(K)_{\tors}\simeq H$. By definition, $\Phi_\Q^\star(4)\subseteq \Phi^\infty(4)$, so our task is to find out what structures in $\Phi^\infty(4)$ also appear in $\Phi_\Q(4)$. We claim that $\Phi_{\mathbb Q}^{\star}(4) = S$ where
\begin{eqnarray*}
S\!\!\!& = &\!\!\!\!\!\left\{ \cC_n \; | \; n=1,\dots,10,12,13,15,16,20,24 \right\} \cup \left\{ \cC_2 \times \cC_{2m} \; | \; m=1,\dots,6,8 \right\} \cup \\
& &  \left\{ \cC_3 \times \cC_{3m} \; | \; m=1,2 \right\} \cup  \left\{ \cC_4 \times \cC_{4m} \; | \; m=1,2 \right\} \cup  \left\{ \cC_5 \times \cC_{5}  \right\} \cup  \left\{ \cC_6 \times \cC_{6}  \right\}.
\end{eqnarray*}
We have examples of torsion structures over quartic fields for all the groups in the list $S$: on one hand, all those groups that appear in $\Phi_\Q(2)$ also appear in $\Phi_\Q(4)$ by extending the corresponding quadratic field to an appropriate biquadratic where the torsion subgroup does not grow (see Lemma 2.2 of \cite{chou}) and, on the other hand, we have examples in Table \ref{ex_4} of the remaining groups that occur over quartics. Therefore it remains to prove that if $H\in \Phi_{\mathbb Q}^{\star}(4)$, then 
$$
H\notin \left\{ \cC_n \; | \; n=11,14,17,18,21,22 \right\} \cup \left\{ \cC_2 \times \cC_{2m} \; | \; m=7,9 \right\} \cup  \left\{ \cC_3 \times \cC_{9}  \right\}.
$$
Indeed, 

\begin{itemize}
\item $H\ne \cC_{11}$, $\cC_{17}$, or  $\cC_{22}$ by Theorem \ref{teo}, part (\ref{t2}), since either $11$ or $17$ would divide $|H|$.
\item $H\ne \cC_{14}$, or $\cC_2 \times \cC_{14}$ by Theorem \ref{teo}, part (\ref{t3}).
\item $H\ne \cC_{21}$ by Theorem \ref{teo}, part (\ref{t21}).
\item $H\ne \cC_{18}$, $\cC_2 \times \cC_{18}$, or $\cC_3 \times \cC_{9}$ by Theorem \ref{teo}, part (\ref{t10}). 
\end{itemize}
This concludes the determination of $\Phi_\Q^\star(4)$. It remains to determine $\Phi_\Q^\infty(4)$, i.e., those structures that occur for infinitely many elliptic curves over $\Q$, that are non-isomorphic (over $\overline{\Q}$). Comparing the list $\Phi_\Q^\star(4)$ and Theorem 1.2 of \cite{chou}, all but three structures ($\cC_{15}$, $\cC_{20}$, and $\cC_{24}$) of appear over Galois quartics, and Chou has shown that each one of those appears infinitely often over $\Q$. Hence, it remains to see what happens in the three remaining structures. Our Theorem \ref{c20c24} shows that $\cC_{20}$ and $\cC_{24}$ occur infinitely often, and Theorem \ref{c15} shows that $\cC_{15}$ occurs only for $4$ distinct $j$-invariants, as claimed. Hence, $\Phi_\Q^\infty(4) = \Phi_\Q^\star(4) \setminus \{\cC_{15}\}$ and this concludes the proof of the theorem.
\end{proof}

\begin{proof}[Proof of Theorem \ref{main2}]
The groups $H\in\Phi^{\star}_\Q(4)$  that do not appear in some $\Phi^{\star}_\Q(4,G)$ for any $G\in\Phi(1)$, with $G\subseteq H$,  can be ruled out using Theorem \ref{teo}. In Table \ref{table1} below, for each group $G$ at the top of a column, we indicate what groups $H$ (in each row) may appear, and indicate
\begin{itemize}
\item with (\ref{t1})--(\ref{t11}), which part of Theorem \ref{teo} is used to prove that the pair $(G,H)$ cannot appear, 
\item with $-$, if the case is ruled out because $G \not\subset H$, 
\item with a $\checkmark$, if the case is possible and, in fact, it occurs. There are three types of check marks in Table \ref{table1}:
\begin{itemize}  
\item $\checkmark$ (without a subindex) means that $G=H$. Note that for any $d\geq 1$, and any elliptic curve $E/\Q$ with $E(\Q)_\text{tors}\simeq G$, there is always an extension $K/\Q$ of degree $d$ such that $E(K)_\text{tors}\simeq E(\Q)_\text{tors}$ (and, in fact, this is the case for almost all degree $d$ extensions).
\item $\checkmark_{\!\!\!2}$ means that the structure $H$ occurs already over a quadratic field, and examples are already listed in Table 2 of \cite{GJT14}. Since $H$ occurs over a quadratic field $F$, it also occurs over quartics by extending $F$ to an appropriate biquadratic $K$ where the torsion does not grow any further.
\item  $\checkmark_{\!\!\!4}$ means that $H$ can be achieved over a quartic field $K$ but not over an intermediate quadratic field, and we have collected examples of curves and quartic fields in Table \ref{ex_4}.
\end{itemize}
\end{itemize}
\end{proof}

{\footnotesize
\renewcommand{\arraystretch}{1.2}
\begin{longtable}[h]{|c|c|c|c|c|c|c|c|c|c|c|c|c|c|c|c|}
\hline
\backslashbox{$H$}{$G$}
 & $\cC_1$ & $\cC_2$ & $\cC_3$ & $\cC_4$ & $\cC_5$ & $\cC_6$ & $\cC_7$ & $\cC_8$ & $\cC_9$ & $\cC_{10}$ & $\cC_{12}$ & $\cC_2 \times \cC_2$ & $\cC_2 \times \cC_4$& $\cC_2 \times \cC_6$& $\cC_2 \times \cC_8$\\
\hline
\endfirsthead
\hline
\backslashbox{$H$}{$G$} 
& $\cC_1$ & $\cC_2$ & $\cC_3$ & $\cC_4$ & $\cC_5$ & $\cC_6$ & $\cC_7$ & $\cC_8$ & $\cC_9$ & $\cC_{10}$ & $\cC_{12}$ & $\cC_2 \times \cC_2$ & $\cC_2 \times \cC_4$& $\cC_2 \times \cC_6$& $\cC_2 \times \cC_8$\\
\hline
\endhead
\endfoot
\endlastfoot
$\cC_1$ & $\checkmark$ & $-$ & $-$ & $-$ & $-$ & $-$ & $-$ & $-$ & $-$ & $-$ & $-$ & $-$ & $-$ & $-$ & $-$ \\
\hline
$\cC_2$ & (\ref{t1}) & $\checkmark$ & $-$ & $-$ & $-$ & $-$ & $-$ & $-$ & $-$ & $-$ & $-$ & $-$ & $-$ & $-$ & $-$ \\
\hline
$\cC_3$ & $\checkmark_{\!\!\!2}$ & $-$ & $\checkmark$ & $-$ & $-$ & $-$ & $-$ & $-$ & $-$ & $-$ & $-$ & $-$ & $-$ & $-$ & $-$ \\
\hline
$\cC_4$ & (\ref{t1}) & $\checkmark_{\!\!\!2}$ & $-$ & $\checkmark$ & $-$ & $-$ & $-$ & $-$ & $-$ & $-$ & $-$ & $-$ & $-$ & $-$ & $-$\\
\hline
$\cC_5$ & $\checkmark_{\!\!\!2}$ & $-$ & $-$ & $-$ & $\checkmark$ & $-$ & $-$ & $-$ & $-$ & $-$ & $-$ & $-$ & $-$ & $-$ & $-$ \\
\hline
$\cC_6$ & (\ref{t1}) & $\checkmark_{\!\!\!2}$ & (\ref{t1}) & $-$ & $-$ & $\checkmark$ & $-$ & $-$ & $-$ & $-$ & $-$ & $-$ & $-$ & $-$ & $-$\\
\hline
$\cC_7$ & $\checkmark_{\!\!\!2}$ & $-$ & $-$ & $-$ & $-$ & $-$ & $\checkmark$ & $-$ & $-$ & $-$ & $-$& $-$ & $-$ & $-$ & $-$ \\
\hline
$\cC_8$ & (\ref{t1}) & $\checkmark_{\!\!\!2}$ & $-$ & $\checkmark_{\!\!\!2}$ & $-$ & $-$ & $-$ & $\checkmark$ & $-$ & $-$ & $-$  & $-$ & $-$ & $-$ & $-$ \\
\hline
$\cC_9$ & $\checkmark_{\!\!\!2}$ & $-$ & (\ref{t11}) & $-$ & $-$ & $-$ & $-$ & $-$ & $\checkmark$ & $-$ & $-$& $-$ & $-$ & $-$ & $-$\\
\hline
$\cC_{10}$ & (\ref{t1}) & $\checkmark_{\!\!\!2}$ & $-$ & $-$ & (\ref{t1}) & $-$ & $-$ & $-$ & $-$ & $\checkmark$ & $-$ & $-$ & $-$ & $-$ & $-$ \\
\hline
$\cC_{11}$ & (\ref{t2}) & $-$ & $-$ & $-$ & $-$ & $-$ & $-$ & $-$ & $-$ & $-$ & $-$ & $-$ & $-$ & $-$ & $-$ \\\hline
$\cC_{12}$ & (\ref{t1}) & $\checkmark_{\!\!\!2}$ & (\ref{t1}) & $\checkmark_{\!\!\!2}$ & $-$ & $\checkmark_{\!\!\!2}$ & $-$ & $-$ & $-$ & $-$ & $\checkmark$ & $-$ & $-$ & $-$ & $-$\\
\hline
$\cC_{13}$ & $\checkmark_{\!\!\!4}$& $-$ & $-$ & $-$ & $-$ & $-$ & $-$ & $-$ & $-$ & $-$ & $-$ & $-$ & $-$ & $-$ & $-$\\
\hline
$\cC_{14}$ & (\ref{t1}) & {(\ref{t3})} & $-$ & $-$ & $-$ & $-$ & (\ref{t1}) & $-$ & $-$ & $-$ & $-$ & $-$ & $-$ & $-$ & $-$\\
\hline
$\cC_{15}$ & $\checkmark_{\!\!\!4}$ & $-$ & $\checkmark_{\!\!\!2}$ & $-$ & $\checkmark_{\!\!\!2}$ & $-$ & $-$ & $-$ & $-$ & $-$ & $-$ & $-$ & $-$ & $-$ & $-$\\
\hline
$\cC_{16}$ & (\ref{t1}) & $\checkmark_{\!\!\!2}$ & $-$ & $\checkmark_{\!\!\!4}$ & $-$ & $-$ & $-$ & $\checkmark_{\!\!\!2}$ & $-$ & $-$ & $-$  & $-$ & $-$ & $-$ & $-$\\
\hline
$\cC_{17}$ & (\ref{t2}) & $-$ & $-$ & $-$ & $-$ & $-$ & $-$ & $-$ & $-$ & $-$ & $-$  & $-$ & $-$ & $-$ & $-$\\
\hline
$\cC_{18}$ & (\ref{t1}) & (\ref{t10}) & (\ref{t1}) & $-$ & $-$ &(\ref{t10})& $-$ & $-$ & (\ref{t1}) & $-$ & $-$  & $-$ & $-$ & $-$ & $-$\\
\hline
$\cC_{20}$ & (\ref{t1}) & $\checkmark_{\!\!\!4}$ & $-$ & (\ref{t4}) & (\ref{t1}) & $-$ & $-$ & $-$ & $-$ & $\checkmark_{\!\!\!4}$ & $-$  & $-$ & $-$ & $-$ & $-$\\
\hline
$\cC_{21}$ & {(\ref{t21})} & $-$ & {(\ref{t21})} & $-$ & $-$ & $-$ & {(\ref{t21})} & $-$ & $-$ & $-$ & $-$  & $-$ & $-$ & $-$ & $-$\\
\hline
$\cC_{22}$ & (\ref{t1}) & (\ref{t2}) & $-$ & $-$ & $-$ & $-$ & $-$ & $-$ & $-$ & $-$ & $-$  & $-$ & $-$ & $-$ & $-$\\
\hline
$\cC_{24}$ & (\ref{t1}) & $\checkmark_{\!\!\!4}$ & (\ref{t1}) & $\checkmark_{\!\!\!4}$  & $-$ &  $\checkmark_{\!\!\!4}$ & $-$ & (\ref{t5}) & $-$ & $-$ & $\checkmark$  & $-$ & $-$ & $-$ & $-$\\
\hline
$\cC_2 \times \cC_2$ & (\ref{t1}) & $\checkmark_{\!\!\!2}$ & $-$ & $-$ & $-$ & $-$ & $-$ & $-$ & $-$ & $-$ & $-$  & $\checkmark$ & $-$ & $-$ & $-$\\
\hline
$\cC_2 \times \cC_4$ & (\ref{t1}) & $\checkmark_{\!\!\!4}$ & $-$ & $\checkmark_{\!\!\!2}$ & $-$ & $-$ & $-$ & $-$ & $-$ & $-$ & $-$ & $\checkmark_{\!\!\!2}$ & $\checkmark$ & $-$ & $-$\\
\hline
$\cC_2 \times \cC_6$ & (\ref{t1}) & $\checkmark_{\!\!\!2}$ & (\ref{t1}) & $-$ & $-$ & $\checkmark_{\!\!\!2}$ & $-$ & $-$ & $-$ & $-$ & $-$ & $\checkmark_{\!\!\!2}$ & $-$ & $\checkmark$ & $-$\\
\hline
$\cC_2 \times \cC_8$ & (\ref{t1}) & $\checkmark_{\!\!\!4}$ & $-$ & $\checkmark_{\!\!\!2}$ & $-$ & $-$ & $-$ & $\checkmark_{\!\!\!2}$ & $-$ & $-$ & $-$ &$\checkmark_{\!\!\!2}$ & $\checkmark_{\!\!\!2}$ & $-$ & $\checkmark$\\
\hline
$\cC_2 \times \cC_{10}$ & (\ref{t1}) & $\checkmark_{\!\!\!2}$ & $-$ & $-$ & (\ref{t1}) & $-$ & $-$ & $-$ & $-$ & $\checkmark_{\!\!\!2}$ & $-$& (\ref{t6}) & $-$ & $-$ & $-$ \\
\hline
$\cC_2 \times \cC_{12}$ & (\ref{t1}) & $\checkmark_{\!\!\!4}$ & (\ref{t1}) & $\checkmark_{\!\!\!2}$ & $-$ &  $\checkmark_{\!\!\!4}$ & $-$ & $-$ & $-$ & $-$ & $\checkmark_{\!\!\!2}$ & $\checkmark_{\!\!\!2}$ & (\ref{t7}) & $\checkmark$ & $-$ \\
\hline
$\cC_2 \times \cC_{14}$ & (\ref{t1}) & {(\ref{t3})} & $-$ & $-$ & $-$ & $-$ & {(\ref{t3})} & $-$ & $-$ & $-$ & $-$ & {(\ref{t3})} & $-$ & $-$ & $-$ \\
\hline
$\cC_2 \times \cC_{16}$ & (\ref{t1}) & $\checkmark_{\!\!\!4}$ & $-$ & $\checkmark_{\!\!\!4}$ & $-$ & $-$ & $-$ & $\checkmark_{\!\!\!2}$  & $-$ & $-$ & $-$ &$\checkmark_{\!\!\!2}$ & $\checkmark_{\!\!\!2}$ & $-$ & $\checkmark_{\!\!\!2}$ \\
\hline
$\cC_2 \times \cC_{18}$ & (\ref{t1}) & (\ref{t10}) & (\ref{t1}) & $-$ & $-$ & (\ref{t10}) & $-$ & $-$ & (\ref{t1}) & $-$ & $-$ & (\ref{t10}) & $-$ & (\ref{t10}) & $-$ \\
\hline
$\cC_3 \times \cC_3$ & $\checkmark_{\!\!\!4}$ & $-$ & $\checkmark_{\!\!\!2}$ & $-$ & $-$ & $-$ & $-$ & $-$ & $-$ & $-$ & $-$ & $-$ & $-$ & $-$ & $-$\\
\hline
$\cC_3 \times \cC_6$ & (\ref{t1}) & $\checkmark_{\!\!\!4}$ & (\ref{t1}) & $-$ & $-$ & $\checkmark_{\!\!\!2}$ & $-$ & $-$ & $-$ & $-$ & $-$& $-$ & $-$ & $-$ & $-$ \\
\hline
$\cC_3 \times \cC_9$ & (\ref{t10}) & $-$ & (\ref{t10}) & $-$ & $-$ & $-$ & $-$ & $-$ & (\ref{t10}) & $-$ & $-$& $-$ & $-$ & $-$ & $-$ \\
\hline
$\cC_4 \times \cC_4$ & (\ref{t1}) & $\checkmark_{\!\!\!4}$& $-$ & $\checkmark_{\!\!\!2}$ & $-$ & $-$ & $-$ & $-$ & $-$ & $-$ & $-$ & $\checkmark_{\!\!\!4}$ & $\checkmark_{\!\!\!2}$ & $-$ & $-$ \\
\hline
$\cC_4 \times \cC_8$ & (\ref{t1}) & $\checkmark_{\!\!\!4}$ & $-$ & $\checkmark_{\!\!\!4}$ & $-$ & $-$ & $-$ & $\checkmark_{\!\!\!4}$  & $-$ & $-$ & $-$ & $\checkmark_{\!\!\!4}$ & $\checkmark_{\!\!\!4}$ & $-$ & $\checkmark_{\!\!\!4}$ \\
\hline
$\cC_5 \times \cC_5$ & $\checkmark_{\!\!\!4}$ & $-$ & $-$ & $-$ & $\checkmark_{\!\!\!4}$ & $-$ & $-$ & $-$ & $-$ & $-$ & $-$ & $-$ & $-$ & $-$ & $-$ \\
\hline
$\cC_6  \times \cC_6$ & (\ref{t1}) & $\checkmark_{\!\!\!4}$ & (\ref{t1}) & $-$ & $-$ &  $\checkmark_{\!\!\!4}$ & $-$ & $-$ & $-$ & $-$ & $-$ & (\ref{t8}) & $-$ & (\ref{t8}) & $-$ \\
\hline
\caption{The table displays either if the case happens for $G=H$ ($\checkmark$), if it already occurs over a quadratic field ($\checkmark_{\!\!\!2}$), if it occurs over a quartic but not a quadratic ($\checkmark_{\!\!\!4}$), if it is impossible because $G \not\subset H$ ($-$) or if it is ruled out by Theorem \ref{teo} ((\ref{t1})--(\ref{t11})).}\label{table1}
\end{longtable}
}

\section{Examples}
In this section we describe an algorithm to compute the quartic fields $K$ where the torsion grows for a given elliptic curve $E/\Q$. First, we compute $G=E(\Q)_{\tors}$. By Theorem \ref{main2} we know how the set $\Phi^\star_{\Q}(4,G)$ of all possible isomorphism types of $E(K)_\text{tors}$. Then, we compute the possible orders of points belonging to groups $H\in \Phi^\star_{\Q}(4,G)$. Now for each possible order of a torsion point, say $n$, compute the division polynomial $\psi_n(x)$ (note that here $\psi_n(x)$ is divisible by  $\psi_m(x)$ for each divisor $m$ of $n$). Afterwards, we factor each $\psi_n(x)$, and keep only the factors of degree $1$, $2$, or $4$. It follows that the quadratic and quartic fields where the torsion could grow are contained in the compositum of the fields generated by these factors together with the fields generated by the $y$-coordinates corresponding to the roots of these factors. Let us explain this method with two examples.

\begin{example}\label{50a2}
Let $E$ be the elliptic curve \texttt{50a2}, given by the  minimal Weierstrass equation:
$$
E\,:\,y^2 + xy + y = x^3 - 126x - 552.
$$
We compute $E(\Q)_{\tors}\simeq \cC_1$. Then, by Theorem \ref{main2} we only need to compute the division polynomials $\psi_n(x)$ for $n=5,7,9,13$. We check that for $n=7,13$, $\psi_n(x)$ is irreducible over $\Q$ (and degree $>4$). For $n=5,9$ we have the following irreducible factors of degree $1$, $2$, or $4$:
$$
\begin{array}{ccl}
n=5 & & f_5(x)=x^2 + 11x + 29\\
n=9 & & f_9(x)=9x + 57.
\end{array}
$$
Now for $n\in\{5,9\}$, let $\alpha_n$ be a root of $f_n(x)$, $\beta_n$ such that $\beta_n^2 + \alpha_n\beta_n + \beta_n = \alpha_n^3 - 126\alpha_n - 552$ and $K_n=\Q(\alpha_n,\beta_n)$, so that $E$ acquires a point of order $n$ over $K_n$. It remains to compute the degree of $K_5$, $K_9$ and $K_5K_9$, and take only those fields of degree $\le 4$. In our case we obtain:
$$
E(\Q(\sqrt{-3}))\simeq \cC_3 \qquad\mbox{and} \qquad E(\Q(\zeta_5))\simeq \cC_5.
$$ 
Hence, the torsion subgroup of $E(\Q)$ grows over $\Q(\zeta_5)$ and quartic fields containing $\Q(\sqrt{-3})$.
\end{example}

\begin{example}\label{90c4}
Let $E$ be the elliptic curve \texttt{90c4}, given by the minimal Weierstrass equation:
$$
E\,:\,y^2 + xy + y = x^3 - x^2 - 2597x - 50281.
$$
In this case we have $E(\Q)_{\tors}\simeq \cC_2$. Theorem \ref{main2} shows that it suffices to factor the division polynomials $\psi_n(x)$ for $n=3,5,16$. The irreducible factors of degree $1$, $2$, or $4$ corresponding to those division polynomials are:
$$
\begin{array}{ccl}
n=3 & & f_3(x)=x+30\\
n=16 & & f_{16,1}(x)=x + 33,\\     
         & &  f_{16,2}(x)=2x + 51,\\
         & &  f_{16,3}(x)=4x + 117,\\ 
           & &  f_{16,4}(x)=x^2 - 30x - 1729,\\
        & & f_{16,5}(x)=x^4 - 60x^3 - 10314x^2 - 351756x - 3697893,\\
         & & f_{16,6}(x)=2x^4 + 204x^3 + 10233x^2 + 274806x + 2924667,\\
         & & f_{16,7}(x)=x^4 + 132x^3 + 5094x^2 + 59508x - 46089.\\
\end{array}
$$
Doing all the possible compositums of number fields we obtain:
$$
\begin{array}{lcl}
E(\Q(\sqrt{-1}))\simeq \cC_4, & & E(\Q(\sqrt{3},\sqrt{-1}))\simeq \cC_ {12},\\
E(\Q(\sqrt{-6}))\simeq \cC_4, & & E(\Q(\sqrt{3},\sqrt{-2}))\simeq \cC_{12},\\
E(\Q(\sqrt{-3}))\simeq \cC_6, & &      E(\Q(\sqrt{6},\sqrt{-1}))\simeq \cC_2\times\cC_4 \simeq E(\Q(\sqrt[4]{6})) ,\\
E(\Q(\sqrt{6}))\simeq \cC_2\times\cC_2, & &  E(\Q(\sqrt{-3},\sqrt{-2}))\simeq \cC_2\times\cC_6.    
\end{array}
$$
\end{example}
Further examples can be found in Table \ref{ex_4}. Each row shows the label of an elliptic curve $E/\Q$ such that $E(\Q)_{\tors}\simeq G$, in the first column, and $E(K)_{\tors}\simeq H$, in the second column, where $K=\Q(\alpha)$ and $\alpha$ is a root of the irreducible quartic in the third column. Note that these examples correspond to pairs $(G,H)$ such that $G\in\Phi(1)$ and $H\in\Phi_\Q^\star(4,G)$ but $H\notin\Phi_\Q(2,G)$. Examples of curves with $H\in \Phi_\Q(2,G)$ can be found in Table 2 of \cite{GJT14}.

\begin{table}[h]
\caption{Examples of elliptic curves such that $G\in\Phi(1)$ and $H\in\Phi_\Q^\star(4,G)$ but $H\notin \Phi_\Q(2,G)$}\label{ex_4}
\begin{tabular}{|c|c|c|c|}
\hline
$G$ & $H$ & Quartic $K$ & Label of $E/\Q$ \\
\hline
\multirow{4}{*}{$\cC_{1}$}  & $\cC_{13}$ & $x^4 - x^3 - 6 x^2 + x + 1$ &  \texttt{2890d1} \\ \cline{2-4}
  & $\cC_{15}$ & $x^4 - 2 x^3 + 5 x^2 - 4 x + 19$ &  \texttt{50a4} \\ \cline{2-4}
  & $\cC_3\times\cC_3$ & $x^4 - 2 x^3 + 5 x^2 - 4 x + 19$ &  \texttt{175b2} \\ \cline{2-4}
  & $\cC_5\times\cC_5$ & $x^4 + x^3 + x^2 + x + 1$ &  \texttt{275b2} \\ \hline
\multirow{10}{*}{$\cC_{2}$}  & $\cC_{20}$ & $x^4 - 5 x^2 + 10$ &  \texttt{450a4} \\ \cline{2-4}
 & $\cC_{24}$ & $x^4 - 18 x^2 - 15$ &  \texttt{960o3} \\ \cline{2-4}
 & $\cC_2\times\cC_4$ & $x^4 - 5$ &  \texttt{15a5} \\ \cline{2-4}
 & $\cC_2\times\cC_8$ & $x^4 + 1$ &  \texttt{24a6} \\ \cline{2-4}
 & $\cC_2\times\cC_{12}$ & $x^4 - 2 x^3 + 5 x^2 - 4 x + 19$ &  \texttt{30a3} \\ \cline{2-4}
 & $\cC_2\times\cC_{16}$ & $x^4 - 4 x^3 + 17 x^2 - 26 x + 16$ &  \texttt{3150bk1} \\ \cline{2-4}
 & $\cC_3\times\cC_6$ & $x^4 - 2 x^3 + 11 x^2 - 10 x + 4$ &  \texttt{98a4} \\ \cline{2-4}
 & $\cC_4\times\cC_4$ & $x^4 + 1$ &  \texttt{64a4} \\ \cline{2-4}
 & $\cC_4\times\cC_8$ & $x^4 + 9$ &  \texttt{2880r6} \\ \cline{2-4}
 & $\cC_{6}\times\cC_{6}$ & $x^4 - 2 x^3 + 11 x^2 - 10 x + 4$ &  \texttt{98a3} \\ \hline
\multirow{4}{*}{$\cC_{4}$}   & $\cC_{16}$ & $x^4 - x^3 - 4 x^2 + 4 x + 1$ &  \texttt{15a7} \\\cline{2-4}
 & $\cC_{24}$ & $x^4 - 8 x^2 + 10$ &  \texttt{960o8} \\\cline{2-4}
 & $\cC_2\times\cC_{16}$ & $x^4 - 4 x^3 + 17 x^2 - 26 x + 16$ &  \texttt{1470k1} \\ \cline{2-4}
  & $\cC_4\times\cC_8$ & $x^4 + 9$ &  \texttt{240d6} \\ \hline
 $\cC_{5} $ & $\cC_5\times\cC_5$ & $x^4 - x^3 + x^2 - x + 1$ &  \texttt{11a1} \\ \hline
\multirow{3}{*}{$\cC_{6}$} & $\cC_{24}$ & $x^4 - 8 x^2 + 10$ &  \texttt{90c8} \\ \cline{2-4}
 & $\cC_2\times\cC_{12}$ & $x^4 - 2 x^3 + 5 x^2 - 4 x + 19$ &  \texttt{30a1} \\ \cline{2-4}
 & $\cC_{6}\times\cC_{6}$ & $x^4 - 2 x^3 + 11 x^2 - 10 x + 4$ &  \texttt{14a1} \\ \hline
\multirow{2}{*}{$\cC_{8}$}  & $\cC_2\times\cC_{16}$ & $x^4 - 4 x^3 + 17 x^2 - 26 x + 16$ &  \texttt{210e1} \\ \cline{2-4}
   & $\cC_4\times\cC_8$ & $x^4 + 9$ &  \texttt{15a4} \\ \hline
 $\cC_{10} $ & $\cC_{20}$ & $x^4 - 2 x^3 + x^2 + 2$ &  \texttt{66c1} \\ \hline
 $\cC_{12} $ & $\cC_{24}$ & $x^4 - 18 x^2 - 15$ &  \texttt{90c3} \\ \hline
\multirow{3}{*}{$ \cC_2\times\cC_2 $}    & $\cC_2\times\cC_{16}$ & $x^4 - x^3 - 4 x^2 + 4 x + 1$ &  \texttt{75b2} \\ \cline{2-4}
  & $\cC_4\times\cC_4$ & $x^4 - 2 x^3 + x^2 + 5$ &  \texttt{15a2} \\ \cline{2-4}
   & $\cC_4\times\cC_8$ & $x^4 - 2 x^3 + x^2 + 5$ &  \texttt{75b3} \\ \hline
\multirow{2}{*}{$ \cC_2\times\cC_4 $} & $\cC_2\times\cC_{16}$ & $x^4 - x^3 - 4 x^2 + 4 x + 1$ &  \texttt{15a3} \\ \cline{2-4}
 & $\cC_4\times\cC_8$ & $x^4 - 2 x^3 + x^2 + 5$ &  \texttt{15a1} \\ \hline
\multirow{2}{*}{$ \cC_2\times\cC_8 $} & $\cC_2\times\cC_{16}$ & $x^4 - 2 x^3 - 11 x^2 + 12 x + 186$ &  \texttt{210e2} \\ \cline{2-4}
 & $\cC_4\times\cC_8$ & $x^4 - 2 x^3 + 7 x^2 - 6 x + 2$ &  \texttt{210e2} \\ \hline
\end{tabular}
\end{table}

\section{Computations}\label{appendix}
Let $G\in \Phi(1)$ and let $d$ be a positive integer.  We define the set
$$
\mathcal{H}_{\Q}(d,G) = \{ S_1,...,S_n \}
$$
where $S_i= \left[ H_1,...,H_m \right]$ is a list of groups $H_j \in \Phi_{\mathbb Q}(d,G) \setminus \{ G \}$, such that, for each $i=1,\ldots,n$, there exists an elliptic curve $E_i$ defined over ${\mathbb Q}$ that satisfies the following properties:
\begin{itemize}
\item $E_i({\mathbb Q})_{\text{tors}} \simeq G$, and
\item There are number fields $K_1,...,K_m$ (non--isomorphic pairwise) of degree dividing $d$ with $E_i ( K_j )_\text{tors} \simeq H_j$, for all $j=1,...,m$;  and for each $j$ there does not exist $K'_j\subset K_j$ such that $E_i ( K'_j )_\text{tors} \simeq H_j$.
\end{itemize}
We are allowing the possibility of two (or more) of the $H_j$ being isomorphic. 

Note that a similar definition was first introduced in \cite{GJT15} for $d=2$ and generalized in \cite{GJNT15}. The second condition is a little bit different here, because of a new behavior that appears only for $d=4$ but not for $d=2,3$ (since they are primes), namely the existence of intermediate fields. For example, let $E$ be the elliptic curve \texttt{50a2}. Then $E({\mathbb Q})_{\text{tors}} \simeq \cC_1$ and $E({\mathbb Q(\sqrt{-3})})_{\text{tors}} \simeq \cC_3$ (see Example \ref{50a2}). In particular,  $E({\mathbb Q(\sqrt{-3},\sqrt{d})})_{\text{tors}} \simeq \cC_3$ for any squarefree integer $d\ne -3$. For this reason we have made the above change in the definition of $\mathcal{H}_{\Q}(d,G)$.

The sets $\mathcal{H}_{\Q}(d,G)$ have been determined for $d=2,3$ and for any $G\in \Phi(1)$ in \cite{GJT15,GJNT15}. In order to guess what $\mathcal{H}_{\Q}(4,G)$ may look like, we carried out an exhaustive computation in \texttt{Magma} \cite{magma} for all elliptic curves over $\Q$ with conductor less than $350.000 $ from \cite{cremonaweb} (a total of $2.188.263$ elliptic curves) but restricting to the non-sporadic case. That is, we have tried to compute the sets $\mathcal{H}^\star_{\Q}(4,G)$, which are similarly defined to the sets $\mathcal{H}_{\Q}(4,G)$  but restricting our attention to $H_j\in  \Phi^\star_{\mathbb Q}(4,G)$.

Moreover, it has been determined the maximum number of quadratic \cite{GJT15,N15b} and cubic \cite{GJNT15} fields where the torsion could grow. In the case of number fields of non prime degree the situation changes. As the example above of the elliptic curve \texttt{50a2} shows, there could be infinitely many non-isomorphic number fields where the torsion grows.  Let us define
$$
h_{\Q}(d)=\max_{G \in \Phi(1)} \Big\{ \#S \; \Big| \; S \in \mathcal{H}_{\Q}(d,G) \Big\}.
$$
Note that if $d$ is prime, then $h_{\Q}(d)$ coincides with the maximum number of number fields of degree $d$ where the torsion grows for a fixed elliptic curve $E/\Q$. The cases $h_{\Q}(2)=4$ and $h_{\Q}(3)=3$ have been determined in \cite{N15b,GJT15} and \cite{GJNT15}  respectively. Our computations (see Table \ref{tablagrande}) and in particular Example \ref{90c4} show that
$$
h_{\Q}(4)\ge 9.
$$

Table \ref{tablagrande} gives  all the torsion configurations over quartic fields (sets in $\mathcal{H}^\star_{\Q}(4,G)$ for any $G\in \Phi(1)$) that we have found. We have found {$130$} possible configurations. However, we have not tried to determine that those are all the possible cases. But note that the largest conductor where we needed to complete the table was {$14.400$}, far from $350.000$.  

In Table \ref{tablagrande}, the third column provides an elliptic curve $E/\Q$ with minimal conductor such that:
\begin{itemize}
\item the first column is $G\simeq E(\Q)_{\text{tors}} \in \Phi (1)$; 
\item the second is a torsion configuration $\left[ H_1,...,H_m \right]$, where $H_j \in \Phi^\star_{\mathbb Q}(4,G) \setminus \{ G \}$, such that there are number fields $K_1,...,K_m$ (non--isomorphic pairwise) of degree dividing $2$ or $4$ with $E ( K_j )_{tors} \simeq H_j$, for all $j=1,...,m$;  and for each $j$ there does not exist $K'_j\subset K_j$ such that $E_i ( K'_j )_{tors} \simeq H_j$.
\end{itemize}
In Table \ref{tablagrande}, we have abbreviated $\cC_n$ by $(n)$, and $\cC_n\times \cC_m$ by $(n,m)$. Moreover, if $H=\cC_n\times \cC_m$ appears for $s$ distinct fields, then we have written $(n,m)^s$ in the table. The corresponding fields $K_j$ for each torsion configuration can be found in the home page of the first author (at \href{http://www.uam.es/enrique.gonzalez.jimenez}{http://www.uam.es/enrique.gonzalez.jimenez}), together with a list of configurations for all curves of conductor up to $350.000$.

\begin{table}
\caption{Torsion configurations over quartic fields}\label{tablagrande}
\begin{tabular}{ccc}
\begin{tabular}{|c|l|c|}
\hline
$G$ & $\mathcal{H}^\star_{\Q}(4,G)$ & Label\\
\hline\hline
\multirow{13}{*}{$(1)$} &   $(3) $ & \texttt{19a2}\\
\cline{2-3}
& $(5) $ & \texttt{11a2}\\
\cline{2-3}
& $(7) $ & \texttt{208d1}\\
\cline{2-3}
& $(9) $ & \texttt{54a2}\\
\cline{2-3}
& $(13) $ & \texttt{2890d1}\\
\cline{2-3}
& $(3)^2 $ & \texttt{121b1}\\
\cline{2-3}
& $(3),(5) $ & \texttt{50a2}\\
\cline{2-3}
& $(3),(15) $ & \texttt{50b3}\\
\cline{2-3}
& $(5)^2 $ & \texttt{99d2}\\
\cline{2-3}
& $(5), (5,5) $ & \texttt{275b2}\\
\cline{2-3}
& $(3)^2,(5) $ & \texttt{338d1}\\
\cline{2-3}
& $(3),(5),(15) $ & \texttt{50a4}\\
\cline{2-3}
& $(3)^2, (3,3) $ & \texttt{175b2}\\
\hline
\hline
\multirow{25}{*}{$(2)$} & $(4), (2,2) $ & \texttt{46a1}\\
\cline{2-3}
& $(4), (2,6) $ & \texttt{36a3}\\
\cline{2-3}
& $(4), (2,10) $ & \texttt{450a3}\\
\cline{2-3}
& $ (2,2), (2,4) $ & \texttt{200b1}\\
\cline{2-3}
& $(4),(10), (2,2) $ & \texttt{66c3}\\
\cline{2-3}
& $(4), (2,2), (2,4) $ & \texttt{49a1}\\
\cline{2-3}
& $(4), (2,2), (2,10) $ & \texttt{1014c2}\\
\cline{2-3}
& $(4), (2,6), (2,12) $ & \texttt{1040g2}\\
\cline{2-3}
& $(8), (2,2), (2,4) $ & \texttt{294f1}\\
\cline{2-3}
& $(4)^2, (2,2), (2,4) $ & \texttt{120b1}\\
\cline{2-3}
& $(4)^2, (2,2), (4,4) $ & \texttt{320a4}\\
\cline{2-3}
& $(4)^2, (2,6), (2,12) $ & \texttt{450g1}\\
\cline{2-3}
& $(4),(6)^2, (2,2) $ & \texttt{726a2}\\
\cline{2-3}
& $(4),(6), (2,2), (2,6) $ & \texttt{14a3}\\
\cline{2-3}
& $(4),(6), (2,6), (6,6) $ & \texttt{98a3}\\
\cline{2-3}
& $(4),(8), (2,2), (2,8) $ & \texttt{45a1}\\
\cline{2-3}
& $(4),(10), (2,2), (2,10) $ & \texttt{150b3}\\
\cline{2-3}
& $(4),(12), (2,2), (2,12) $ & \texttt{30a3}\\
\cline{2-3}
& $(4),(16), (2,2), (2,16) $ & \texttt{3150bk1}\\
\cline{2-3}
& $(6)^2, (2,2), (2,4) $ & \texttt{256a1}\\
\hline
\end{tabular}
&
\begin{tabular}{|c|l|c|}
\hline
$G$ & $\mathcal{H}^\star_{\Q}(4,G)$  & Label\\
\hline\hline
\multirow{22}{*}{$(2)$} & $(6),(12), (2,2), (2,6) $ & \texttt{36a4}\\
\cline{2-3}
& $(8)^2, (2,2), (4,8) $ & \texttt{2880r6}\\
\cline{2-3}
& $(10),(20), (2,2), (2,10) $ & \texttt{450a4}\\
\cline{2-3}
& $(4)^2,(8), (2,2), (2,4) $ & \texttt{33a2}\\
\cline{2-3}
&  $(4)^2,(8), (2,2), (4,4) $ & \texttt{64a4}\\
\cline{2-3}
& $(4)^2, (2,2), (2,4)^2 $ & \texttt{33a4}\\
\cline{2-3}
& $(4)^2, (2,6), (2,12)^2 $ & \texttt{960o7}\\
\cline{2-3}
& $(4),(6), (2,2), (2,4), (2,6) $ & \texttt{130a4}\\
\cline{2-3}
& $(4),(8),(12), (2,2), (2,12) $ & \texttt{960e3}\\
\cline{2-3}
& $(4),(8),(16), (2,2), (2,8) $ & \texttt{63a1}\\
\cline{2-3}
& $(4),(8), (2,2), (2,4), (2,8) $ &  \texttt{144b1}\\
\cline{2-3}
& $(4),(12),(24), (2,2), (2,12) $ & \texttt{960o3}\\
\cline{2-3}
& $(4),(12), (2,2), (2,4), (2,12) $ & \texttt{720j3}\\
\cline{2-3}

& $(4)^2,(8)^2, (2,2), (2,4) $ & \texttt{45a3}\\
\cline{2-3}

& $(4),(8)^2, (2,2), (2,4),(2,8) $ & \texttt{24a6}\\
\cline{2-3}
& $(4),(8)^3, (2,2),(2,8) $ & \texttt{45a6}\\
\cline{2-3}

& $(4)^2,(8), (2,2), (2,4)^2 $ & \texttt{17a3}\\
\cline{2-3}
& $(4),(8),(16)^2, (2,2), (2,8) $ & \texttt{75b1}\\
\cline{2-3}

& $(4),(8)^3,(16), (2,2),  (2,8) $ &   \texttt{75b6}\\
\cline{2-3}

& $(4),(8)^2,(16), (2,2), (2,4), (2,8) $ & \texttt{510e7}\\
\cline{2-3}
& $(4)^2,(8)^2, (2,2), (2,4)^2 $ & \texttt{63a6}\\
\cline{2-3}
& $(4),(6)^2, (2,2), (2,6)^2, (3,6) $ & \texttt{112c3}\\
\cline{2-3}
& $(4),(8),(16)^2, (2,2), (2,4), (2,8) $ & \texttt{1470k3}\\
\cline{2-3}
& $(6)^2,(12), (2,2), (2,6)^2, (3,6) $ & \texttt{98a4}\\
\cline{2-3}

& $(4),(8)^2,(16)^2, (2,2), (2,4), (2,8) $ & \texttt{1680p1}\\
\cline{2-3}

& $(4)^2,(6),(12)^2, (2,2), (2,4), (2,6) $ & \texttt{30a7}\\
\cline{2-3}
& $(4)^2,(8)^4, (2,2), (2,4) $ & \texttt{630c6}\\
\cline{2-3}
& $(4)^2,(8)^4, (2,2), (4,4) $ & \texttt{4410r6}\\
\cline{2-3}
& $(4)^2,(8)^3, (2,2), (2,4)^2 $ & \texttt{15a5}\\
\cline{2-3}
& $(4)^2,(6),(8),(12)^2, (2,2), (2,4), (2,6) $ & \texttt{90c5}\\
\cline{2-3}
& $(4)^2,(6),(12)^2, (2,2), (2,4)^2, (2,6) $ & \texttt{90c4}\\
\hline
\multirow{2}{*}{$(3)$} &  $(15)$ & \texttt{50a1}\\
\cline{2-3}
& $ (3,3) $ & \texttt{19a1} \\
\hline

\end{tabular}
\end{tabular}
\end{table}


\begin{table}
\begin{tabular}{ccc}
\begin{tabular}{|c|l|c|}
\hline
$G$ & $\mathcal{H}^\star_{\Q}(4,G)$ & Label \\
\hline\hline
\hline
\multirow{17}{*}{$(4)$} &  $(8), (2,4) $ & \texttt{33a3}\\
\cline{2-3}
& $(8), (2,8) $ & \texttt{192c6}\\
\cline{2-3}
& $(8), (2,12) $ & \texttt{150c3}\\
\cline{2-3}
& $(8), (4,4) $ & \texttt{40a4}\\
\cline{2-3}
& $ (2,4), (2,8) $ & \texttt{64a3}\\
\cline{2-3}
& $(8), (2,4), (2,8) $ & \texttt{17a4}\\
\cline{2-3}
& $(8), (2,4), (4,4) $ & \texttt{17a1}\\
\cline{2-3}
& $(8), (2,8), (2,16) $ & \texttt{1470k1}\\
\cline{2-3}
& $(8)^2, (2,4), (2,8) $ & \texttt{24a3}\\
\cline{2-3}
& $(8)^2, (2,8), (4,8) $ & \texttt{240d6}\\
\cline{2-3}
& $(8),(12), (2,4), (2,12) $ & \texttt{90c1}\\
\cline{2-3}
& $(12),(24), (2,4), (2,12) $ &  \texttt{960o8}\\
\cline{2-3}
 & $(8)^2,(16), (2,4), (2,8) $ & \texttt{21a4}\\
\cline{2-3}
& $(8)^2,(16)^2, (2,4), (2,8) $ & \texttt{15a7}\\
\cline{2-3}
& $(8)^2, (2,4), (2,8)^2, (4,4) $ & \texttt{195a6}\\
\cline{2-3}
& $(8)^2,(16)^3, (2,4), (2,8) $ & \texttt{1230f4}\\
\cline{2-3}
& $(8)^2,(16)^2, (2,4), (2,8)^2, (4,4)$ & \texttt{210e6}  \\
\hline
\hline
\multirow{2}{*}{$(5)$} & $(15)$ & \texttt{50b1}\\
\cline{2-3}
& $ (5,5) $ & \texttt{11a1}\\
\hline
\hline
\multirow{6}{*}{$(6)$} &  $(12), (2,6) $ & \texttt{14a4}\\
\cline{2-3}
& $(12), (2,6), (2,12) $ & \texttt{130a2}\\
\cline{2-3}
& $(12)^2, (2,6), (2,12) $ & \texttt{30a1}\\
\cline{2-3}
& $(12), (2,6), (3,6), (6,6) $ & \texttt{14a1}\\
\cline{2-3}
& $(12)^2,(24), (2,6), (2,12)$ & \texttt{90c8}\\
\cline{2-3}
& $(12)^2, (2,6), (2,12)^2 $ & \texttt{90c7}\\
\hline
\hline
\multirow{4}{*}{$(8)$} & $(16), (2,8) $ & \texttt{21a3}\\
\cline{2-3}
& $(16), (2,8), (2,16) $ & \texttt{1230f1}\\
\cline{2-3}
& $(16), (2,8), (4,8) $ & \texttt{15a4}\\
\cline{2-3}
& $(16)^2, (2,8), (2,16)$ & \texttt{210e1}\\
\hline
\hline
\multirow{1}{*}{$(10)$} & $(20), (2,10)$ &\texttt{66c1} \\
\hline
\hline
\multirow{1}{*}{$(12)$} & $(24), (2,12) $ &\texttt{90c3} \\
\hline

\end{tabular}

&

\begin{tabular}{|c|l|c|}
\hline
$G$ & $\mathcal{H}^\star_{\Q}(4,G)$  & Label\\
\hline\hline
\multirow{26}{*}{$(2,2)$}  & $(2,4), (4,4) $ & \texttt{64a1}\\ 
\cline{2-3}

& $ (2,4)^3 $ & \texttt{33a1}\\
\cline{2-3}
& $ (2,4)^2, (2,8) $ & \texttt{63a2}\\
\cline{2-3}
& $ (2,4)^2, (2,12) $ & \texttt{960o6}\\
\cline{2-3}
& $ (2,4)^2, (4,4) $ & \texttt{17a2}\\
\cline{2-3}
& $ (2,4)^2, (4,8) $ & \texttt{1200j4}\\
\cline{2-3}
& $ (2,4), (2,8), (4,8) $ & \texttt{75b3}\\
\cline{2-3}

& $ (2,4)^3, (2,6) $ & \texttt{210a6}\\
\cline{2-3}
& $ (2,4)^3, (4,4) $ & \texttt{231a3}\\
\cline{2-3}
& $ (2,4)^2, (2,6), (2,12) $ & \texttt{30a6}\\
\cline{2-3}
& $ (2,4)^2, (2,8), (2,16) $ & \texttt{75b2}\\
\cline{2-3}
& $ (2,4)^2, (2,8), (4,4) $ & \texttt{40a1}\\
\cline{2-3}
& $ (2,4)^2, (2,8), (4,8) $ & \texttt{510e5}\\
\cline{2-3}
& $ (2,4), (2,6), (2,12)^2 $ &\texttt{14400bo6} \\
\cline{2-3}
& $ (2,4)^3, (2,8), (4,4) $ & \texttt{21a2}\\
\cline{2-3}
& $ (2,4)^2, (2,8)^2, (4,4) $ & \texttt{75b5}\\
\cline{2-3}
& $ (2,4), (2,6), (2,12)^3 $ & \texttt{150c6}\\
\cline{2-3}

& $ (2,4)^3, (2,8)^2, (4,4) $ & \texttt{42a3}\\
\cline{2-3}
& $ (2,4)^2, (2,8)^3, (4,4) $ & \texttt{294c2}\\
\cline{2-3}

& $ (2,4)^3, (2,8)^3, (4,4) $ & \texttt{15a2}\\
\cline{2-3}
& $ (2,4)^2, (2,8)^4, (4,4) $ & \texttt{6720cd4}\\
\cline{2-3}
& $ (2,4)^3, (2,8)^4, (4,4)$ & \texttt{210e5} \\
\hline
\hline
\multirow{6}{*}{$(2,4)$} & $ (2,8)^2, (4,4) $ & \texttt{21a1}\\
\cline{2-3}
& $ (2,8)^2, (4,8) $ & \texttt{1230f2}\\
\cline{2-3}
& $ (2,8)^2, (2,16), (4,4) $ & \texttt{15a3}\\
\cline{2-3}
& $ (2,8), (4,4), (4,8) $ & \texttt{15a1}\\
\cline{2-3}
& $ (2,8)^2, (4,4), (4,8)$ & \texttt{210e3}\\
\hline
\hline
\multirow{1}{*}{$(2,6)$} &  $ (2,12)^3 $ & \texttt{30a2}\\
\hline
\hline
$(2,8)$ & $ (2,16)^2, (4,8)$ & \texttt{210e2}\\
\hline
\multicolumn{3}{c}{}\\
\multicolumn{3}{c}{}\\
\end{tabular}
\end{tabular}
\end{table}

\end{document}